\documentclass{amsart}
\usepackage{zach-1.1}
\usepackage{tikz}
    
\newcommand{\crossing}{
    \begin{tikzpicture}[scale=0.2,
                baseline={([yshift=-\the\dimexpr\fontdimen22\textfont2\relax]
                    current bounding box.center)},
        ]
        \clip (-0.2,-1.2) rectangle + (2.4,2.4);
        \draw[thick, orange] (0,-1) to (0.8,-0.2);
\draw[thick, orange] (1.2,0.2) to (2,1);
\draw[thick, orange] (0,1) to (2,-1);\end{tikzpicture}
    }
    
\newcommand{\positiveCrossing}{
    \begin{tikzpicture}[scale=0.2,
                baseline={([yshift=-\the\dimexpr\fontdimen22\textfont2\relax]
                    current bounding box.center)},
        ]
        \clip (-0.2,-1.2) rectangle + (2.4,2.4);
        \begin{scope}[xscale=-1, xshift=-2cm]
\draw[thick, orange] (0,-1) to (0.8,-0.2);
\draw[thick, orange, ->] (1.2,0.2) to (2,1);
\draw[thick, orange, ->] (2,-1) to (0,1);
\end{scope}\end{tikzpicture}
    }
    
\newcommand{\negativeCrossing}{
    \begin{tikzpicture}[scale=0.2,
                baseline={([yshift=-\the\dimexpr\fontdimen22\textfont2\relax]
                    current bounding box.center)},
        ]
        \clip (-0.2,-1.2) rectangle + (2.4,2.4);
        \draw[thick, orange] (0,-1) to (0.8,-0.2);
\draw[thick, orange, ->] (1.2,0.2) to (2,1);
\draw[thick, orange, ->] (2,-1) to (0,1);\end{tikzpicture}
    }

\newcommand{\rI}{
    \begin{tikzpicture}[scale=0.2,
                baseline={([yshift=-\the\dimexpr\fontdimen22\textfont2\relax]
                    current bounding box.center)},
        ]
        \clip (-0.2,-1.2) rectangle + (3.4,2.4);
        \draw[thick, orange, out=90, in=90] (0,-1) to (3,0);
\draw[thick, orange, out=270, in=315] (3,0) to (1.1,-0.1);
\draw[thick, orange, out=135] (0.6,0.4) to (0,1);\end{tikzpicture}
    }
    
\newcommand{\rIa}{
    \begin{tikzpicture}[scale=0.2,
                baseline={([yshift=-\the\dimexpr\fontdimen22\textfont2\relax]
                    current bounding box.center)},
        ]
        \clip (-0.2,-1.2) rectangle + (3.4,2.4);
        \draw[thick, orange, out=45, in=270] (0,-1) to (0.5,0);
\draw[thick, orange, out=90, in=315] (0.5,0) to (0,1);
\draw[thick, orange, out=90, in=90] (1,0) to (3,0);
\draw[thick, orange, out=270, in=270] (3,0) to (1,0);\end{tikzpicture}
    }
    
\newcommand{\rIb}{
    \begin{tikzpicture}[scale=0.2,
                baseline={([yshift=-\the\dimexpr\fontdimen22\textfont2\relax]
                    current bounding box.center)},
        ]
        \clip (-0.2,-1.2) rectangle + (3.4,2.4);
        \draw[thick, orange, out=90, in=135] (0,-1) to (1, -0.4);
\draw[thick, orange, out=315, in=270] (1, -0.4) to (3, 0);
\draw[thick, orange, out=90, in=45] (3,0) to (1, 0.4);
\draw[thick, orange, out=225, in=270] (1, 0.4) to (0,1);\end{tikzpicture}
    }
    
\newcommand{\rIaam}{
    \begin{tikzpicture}[scale=0.2,
                baseline={([yshift=-\the\dimexpr\fontdimen22\textfont2\relax]
                    current bounding box.center)},
        ]
        \clip (-0.2,-2.4) rectangle + (3.4,3.8);
        \draw[thick, orange, out=45, in=270] (0,-1) to (0.5,0);
\draw[thick, orange, out=90, in=315] (0.5,0) to (0,1);
\draw[thick, orange, out=90, in=90] (1,0) to (3,0);
\draw[thick, orange, out=270, in=270] (3,0) to (1,0);

\node[] at (0.25,-1.6) {\tiny \color{blue} $a$};
\node[] at (2,-1.6) {\tiny \color{blue} \textminus};\end{tikzpicture}
    }
    
\newcommand{\rIba}{
    \begin{tikzpicture}[scale=0.2,
                baseline={([yshift=-\the\dimexpr\fontdimen22\textfont2\relax]
                    current bounding box.center)},
        ]
        \clip (-0.2,-2.4) rectangle + (3.4,3.8);
        \draw[thick, orange, out=90, in=135] (0,-1) to (1, -0.4);
\draw[thick, orange, out=315, in=270] (1, -0.4) to (3, 0);
\draw[thick, orange, out=90, in=45] (3,0) to (1, 0.4);
\draw[thick, orange, out=225, in=270] (1, 0.4) to (0,1);

\node[] at (1.5,-1.6) {\tiny \color{blue} $a$};\end{tikzpicture}
    }
    
\newcommand{\rII}{
    \begin{tikzpicture}[scale=0.2,
                baseline={([yshift=-\the\dimexpr\fontdimen22\textfont2\relax]
                    current bounding box.center)},
        ]
        \clip (-0.2,-1.2) rectangle + (5.4,2.4);
        \draw[thick, orange, out=90, in=90] (0,-1) to (5,-1);
\draw[thick, orange, out=135, in=-45] (0.6,0.4) to (0,1);
\draw[thick, orange, out=225, in=315] (3.9,-0.1) to (1.1,-0.1);
\draw[thick, orange, out=45, in=225] (4.4,0.4) to (5,1);\end{tikzpicture}
    }
    
\newcommand{\rIIaa}{
    \begin{tikzpicture}[scale=0.2,
                baseline={([yshift=-\the\dimexpr\fontdimen22\textfont2\relax]
                    current bounding box.center)},
        ]
        \clip (-0.2,-1.2) rectangle + (4.4,2.4);
        \input{diagrams/r2aa}\end{tikzpicture}
    }
    
\newcommand{\rIIab}{
    \begin{tikzpicture}[scale=0.2,
                baseline={([yshift=-\the\dimexpr\fontdimen22\textfont2\relax]
                    current bounding box.center)},
        ]
        \clip (-0.2,-1.2) rectangle + (4.4,2.4);
        \draw[thick, orange, out=45, in=270] (0,-1) to (0.5,0);
\draw[thick, orange, out=90, in=315] (0.5,0) to (0,1);
\draw[thick, orange, out=90, in=90] (1,0) to (3,0);
\draw[thick, orange, out=270, in=270] (3,0) to (1,0);
\begin{scope}[xscale=-1, xshift=-4cm]
    \draw[thick, orange, out=45, in=270] (0,-1) to (0.5,0);
    \draw[thick, orange, out=90, in=315] (0.5,0) to (0,1);
\end{scope}\end{tikzpicture}
    }
    
\newcommand{\rIIba}{
    \begin{tikzpicture}[scale=0.2,
                baseline={([yshift=-\the\dimexpr\fontdimen22\textfont2\relax]
                    current bounding box.center)},
        ]
        \clip (-0.2,-1.2) rectangle + (4.4,2.4);
        \input{diagrams/r2ba}\end{tikzpicture}
    }
    
\newcommand{\rIIbb}{
    \begin{tikzpicture}[scale=0.2,
                baseline={([yshift=-\the\dimexpr\fontdimen22\textfont2\relax]
                    current bounding box.center)},
        ]
        \clip (-0.2,-1.2) rectangle + (4.4,2.4);
        \draw[thick, orange, out=90, in=135] (0,-1) to (1, -0.4);
\draw[thick, orange, out=315, in=270] (1, -0.4) to (3, 0);
\draw[thick, orange, out=90, in=45] (3,0) to (1, 0.4);
\draw[thick, orange, out=225, in=270] (1, 0.4) to (0,1);
\begin{scope}[xscale=-1, xshift=-4cm]
    \draw[thick, orange, out=45, in=270] (0,-1) to (0.5,0);
    \draw[thick, orange, out=90, in=315] (0.5,0) to (0,1);
\end{scope}\end{tikzpicture}
    }
    
\newcommand{\zeroRes}{
    \begin{tikzpicture}[scale=0.2,
                baseline={([yshift=-\the\dimexpr\fontdimen22\textfont2\relax]
                    current bounding box.center)},
        ]
        \clip (-0.2,-1.2) rectangle + (2.4,2.4);
        \draw[thick, orange, out=-45, in=-135] (0,1) to (2,1);
\draw[thick, orange, out=45, in=135] (0,-1) to (2,-1);\end{tikzpicture}
    }
    
\newcommand{\oneRes}{
    \begin{tikzpicture}[scale=0.2,
                baseline={([yshift=-\the\dimexpr\fontdimen22\textfont2\relax]
                    current bounding box.center)},
        ]
        \clip (-0.2,-1.2) rectangle + (2.4,2.4);
        \draw[thick, orange, out=45, in=-45] (0,-1) to (0,1);
\draw[thick, orange, out=135, in=-135] (2,-1) to (2,1);\end{tikzpicture}
    }
    
\newcommand{\hopfLink}{
    \begin{tikzpicture}[scale=0.4,
                baseline={([yshift=-\the\dimexpr\fontdimen22\textfont2\relax]
                    current bounding box.center)},
        ]
        \clip (-0.2,-0.7) rectangle + (2.2,1.4);
        \input{diagrams/hopfLink}\end{tikzpicture}
    }
    
\newcommand{\hopfLinkaa}{
    \begin{tikzpicture}[scale=0.4,
                baseline={([yshift=-\the\dimexpr\fontdimen22\textfont2\relax]
                    current bounding box.center)},
        ]
        \clip (-0.2,-0.7) rectangle + (2.2,1.4);
        \input{diagrams/hopfLinkaa}\end{tikzpicture}
    }
    
\newcommand{\hopfLinkab}{
    \begin{tikzpicture}[scale=0.4,
                baseline={([yshift=-\the\dimexpr\fontdimen22\textfont2\relax]
                    current bounding box.center)},
        ]
        \clip (-0.2,-0.7) rectangle + (2.2,1.4);
        \input{diagrams/hopfLinkab}\end{tikzpicture}
    }
    
\newcommand{\hopfLinkba}{
    \begin{tikzpicture}[scale=0.4,
                baseline={([yshift=-\the\dimexpr\fontdimen22\textfont2\relax]
                    current bounding box.center)},
        ]
        \clip (-0.2,-0.7) rectangle + (2.2,1.4);
        \input{diagrams/hopfLinkba}\end{tikzpicture}
    }
    
\newcommand{\hopfLinkbb}{
    \begin{tikzpicture}[scale=0.4,
                baseline={([yshift=-\the\dimexpr\fontdimen22\textfont2\relax]
                    current bounding box.center)},
        ]
        \clip (-0.2,-0.7) rectangle + (2.2,1.4);
        \input{diagrams/hopfLinkbb}\end{tikzpicture}
    }
    
\newcommand{\annularHopfLink}{
    \begin{tikzpicture}[scale=0.4,
                baseline={([yshift=-\the\dimexpr\fontdimen22\textfont2\relax]
                    current bounding box.center)},
        ]
        \clip (-0.2,-0.7) rectangle + (2.2,1.4);
        \input{diagrams/annularHopfLink}\end{tikzpicture}
    }
    
\newcommand{\annularHopfLinkaa}{
    \begin{tikzpicture}[scale=0.4,
                baseline={([yshift=-\the\dimexpr\fontdimen22\textfont2\relax]
                    current bounding box.center)},
        ]
        \clip (-0.2,-0.7) rectangle + (2.2,1.4);
        \input{diagrams/annularHopfLinkaa}\end{tikzpicture}
    }
    
\newcommand{\annularHopfLinkab}{
    \begin{tikzpicture}[scale=0.4,
                baseline={([yshift=-\the\dimexpr\fontdimen22\textfont2\relax]
                    current bounding box.center)},
        ]
        \clip (-0.2,-0.7) rectangle + (2.2,1.4);
        \input{diagrams/annularHopfLinkab}\end{tikzpicture}
    }
    
\newcommand{\annularHopfLinkba}{
    \begin{tikzpicture}[scale=0.4,
                baseline={([yshift=-\the\dimexpr\fontdimen22\textfont2\relax]
                    current bounding box.center)},
        ]
        \clip (-0.2,-0.7) rectangle + (2.2,1.4);
        \input{diagrams/annularHopfLinkba}\end{tikzpicture}
    }
    
\newcommand{\annularHopfLinkbb}{
    \begin{tikzpicture}[scale=0.4,
                baseline={([yshift=-\the\dimexpr\fontdimen22\textfont2\relax]
                    current bounding box.center)},
        ]
        \clip (-0.2,-0.7) rectangle + (2.2,1.4);
        \input{diagrams/annularHopfLinkbb}\end{tikzpicture}
    }
    
\newcommand{\multiHopfLink}{
    \begin{tikzpicture}[scale=0.4,
                baseline={([yshift=-\the\dimexpr\fontdimen22\textfont2\relax]
                    current bounding box.center)},
        ]
        \clip (-0.2,-0.7) rectangle + (2.2,1.4);
        \input{diagrams/multiHopfLink}\end{tikzpicture}
    }
    
\newcommand{\multiHopfLinkaa}{
    \begin{tikzpicture}[scale=0.4,
                baseline={([yshift=-\the\dimexpr\fontdimen22\textfont2\relax]
                    current bounding box.center)},
        ]
        \clip (-0.2,-0.7) rectangle + (2.2,1.4);
        \input{diagrams/multiHopfLinkaa}\end{tikzpicture}
    }
    
\newcommand{\multiHopfLinkab}{
    \begin{tikzpicture}[scale=0.4,
                baseline={([yshift=-\the\dimexpr\fontdimen22\textfont2\relax]
                    current bounding box.center)},
        ]
        \clip (-0.2,-0.7) rectangle + (2.2,1.4);
        \input{diagrams/multiHopfLinkab}\end{tikzpicture}
    }
    
\newcommand{\multiHopfLinkba}{
    \begin{tikzpicture}[scale=0.4,
                baseline={([yshift=-\the\dimexpr\fontdimen22\textfont2\relax]
                    current bounding box.center)},
        ]
        \clip (-0.2,-0.7) rectangle + (2.2,1.4);
        \input{diagrams/multiHopfLinkba}\end{tikzpicture}
    }
    
\newcommand{\multiHopfLinkbb}{
    \begin{tikzpicture}[scale=0.4,
                baseline={([yshift=-\the\dimexpr\fontdimen22\textfont2\relax]
                    current bounding box.center)},
        ]
        \clip (-0.2,-0.7) rectangle + (2.2,1.4);
        \input{diagrams/multiHopfLinkbb}\end{tikzpicture}
    }
    
\newcommand{\rIII}{
    \begin{tikzpicture}[scale=0.25,
                baseline={([yshift=-\the\dimexpr\fontdimen22\textfont2\relax]
                    current bounding box.center)},
        ]
        \clip (-0.2,-0.2) rectangle + (3.4,3.4);
        \input{diagrams/r3}\end{tikzpicture}
    }
    
\newcommand{\rIIIaaa}{
    \begin{tikzpicture}[scale=0.25,
                baseline={([yshift=-\the\dimexpr\fontdimen22\textfont2\relax]
                    current bounding box.center)},
        ]
        \clip (-0.2,-0.2) rectangle + (3.4,3.4);
        \input{diagrams/r3aaa}\end{tikzpicture}
    }
    
\newcommand{\rIIIbaa}{
    \begin{tikzpicture}[scale=0.25,
                baseline={([yshift=-\the\dimexpr\fontdimen22\textfont2\relax]
                    current bounding box.center)},
        ]
        \clip (-0.2,-0.2) rectangle + (3.4,3.4);
        \input{diagrams/r3baa}\end{tikzpicture}
    }
    
\newcommand{\rIIIaba}{
    \begin{tikzpicture}[scale=0.25,
                baseline={([yshift=-\the\dimexpr\fontdimen22\textfont2\relax]
                    current bounding box.center)},
        ]
        \clip (-0.2,-0.2) rectangle + (3.4,3.4);
        \input{diagrams/r3aba}\end{tikzpicture}
    }
    
\newcommand{\rIIIaab}{
    \begin{tikzpicture}[scale=0.25,
                baseline={([yshift=-\the\dimexpr\fontdimen22\textfont2\relax]
                    current bounding box.center)},
        ]
        \clip (-0.2,-0.2) rectangle + (3.4,3.4);
        \input{diagrams/r3aab}\end{tikzpicture}
    }
    
\newcommand{\rIIIbba}{
    \begin{tikzpicture}[scale=0.25,
                baseline={([yshift=-\the\dimexpr\fontdimen22\textfont2\relax]
                    current bounding box.center)},
        ]
        \clip (-0.2,-0.2) rectangle + (3.4,3.4);
        \input{diagrams/r3bba}\end{tikzpicture}
    }
    
\newcommand{\rIIIbab}{
    \begin{tikzpicture}[scale=0.25,
                baseline={([yshift=-\the\dimexpr\fontdimen22\textfont2\relax]
                    current bounding box.center)},
        ]
        \clip (-0.2,-0.2) rectangle + (3.4,3.4);
        \input{diagrams/r3bab}\end{tikzpicture}
    }
    
\newcommand{\rIIIabb}{
    \begin{tikzpicture}[scale=0.25,
                baseline={([yshift=-\the\dimexpr\fontdimen22\textfont2\relax]
                    current bounding box.center)},
        ]
        \clip (-0.2,-0.2) rectangle + (3.4,3.4);
        \input{diagrams/r3abb}\end{tikzpicture}
    }
    
\newcommand{\rIIIbbb}{
    \begin{tikzpicture}[scale=0.25,
                baseline={([yshift=-\the\dimexpr\fontdimen22\textfont2\relax]
                    current bounding box.center)},
        ]
        \clip (-0.2,-0.2) rectangle + (3.4,3.4);
        \input{diagrams/r3bbb}\end{tikzpicture}
    }
    
\newcommand{\rIIIx}{
    \begin{tikzpicture}[scale=0.25,
                baseline={([yshift=-\the\dimexpr\fontdimen22\textfont2\relax]
                    current bounding box.center)},
        ]
        \clip (-0.2,-0.2) rectangle + (3.4,3.4);
        \input{diagrams/r3x}\end{tikzpicture}
    }
    
\newcommand{\rIIIxaaa}{
    \begin{tikzpicture}[scale=0.25,
                baseline={([yshift=-\the\dimexpr\fontdimen22\textfont2\relax]
                    current bounding box.center)},
        ]
        \clip (-0.2,-0.2) rectangle + (3.4,3.4);
        \input{diagrams/r3xaaa}\end{tikzpicture}
    }
    
\newcommand{\rIIIxbaa}{
    \begin{tikzpicture}[scale=0.25,
                baseline={([yshift=-\the\dimexpr\fontdimen22\textfont2\relax]
                    current bounding box.center)},
        ]
        \clip (-0.2,-0.2) rectangle + (3.4,3.4);
        \input{diagrams/r3xbaa}\end{tikzpicture}
    }
    
\newcommand{\rIIIxaba}{
    \begin{tikzpicture}[scale=0.25,
                baseline={([yshift=-\the\dimexpr\fontdimen22\textfont2\relax]
                    current bounding box.center)},
        ]
        \clip (-0.2,-0.2) rectangle + (3.4,3.4);
        \input{diagrams/r3xaba}\end{tikzpicture}
    }
    
\newcommand{\rIIIxaab}{
    \begin{tikzpicture}[scale=0.25,
                baseline={([yshift=-\the\dimexpr\fontdimen22\textfont2\relax]
                    current bounding box.center)},
        ]
        \clip (-0.2,-0.2) rectangle + (3.4,3.4);
        \input{diagrams/r3xaab}\end{tikzpicture}
    }
    
\newcommand{\rIIIxbba}{
    \begin{tikzpicture}[scale=0.25,
                baseline={([yshift=-\the\dimexpr\fontdimen22\textfont2\relax]
                    current bounding box.center)},
        ]
        \clip (-0.2,-0.2) rectangle + (3.4,3.4);
        \input{diagrams/r3xbba}\end{tikzpicture}
    }
    
\newcommand{\rIIIxbab}{
    \begin{tikzpicture}[scale=0.25,
                baseline={([yshift=-\the\dimexpr\fontdimen22\textfont2\relax]
                    current bounding box.center)},
        ]
        \clip (-0.2,-0.2) rectangle + (3.4,3.4);
        \input{diagrams/r3xbab}\end{tikzpicture}
    }
    
\newcommand{\rIIIxabb}{
    \begin{tikzpicture}[scale=0.25,
                baseline={([yshift=-\the\dimexpr\fontdimen22\textfont2\relax]
                    current bounding box.center)},
        ]
        \clip (-0.2,-0.2) rectangle + (3.4,3.4);
        \input{diagrams/r3xabb}\end{tikzpicture}
    }
    
\newcommand{\rIIIxbbb}{
    \begin{tikzpicture}[scale=0.25,
                baseline={([yshift=-\the\dimexpr\fontdimen22\textfont2\relax]
                    current bounding box.center)},
        ]
        \clip (-0.2,-0.2) rectangle + (3.4,3.4);
        \input{diagrams/r3xbbb}\end{tikzpicture}
    }
    
\newcommand{\rIIIb}{
    \begin{tikzpicture}[scale=0.25,
                baseline={([yshift=-\the\dimexpr\fontdimen22\textfont2\relax]
                    current bounding box.center)},
        ]
        \clip (-0.2,-0.2) rectangle + (3.4,3.4);
        \input{diagrams/r3b}\end{tikzpicture}
    }
    
\newcommand{\rIIIxb}{
    \begin{tikzpicture}[scale=0.25,
                baseline={([yshift=-\the\dimexpr\fontdimen22\textfont2\relax]
                    current bounding box.center)},
        ]
        \clip (-0.2,-0.2) rectangle + (3.4,3.4);
        \input{diagrams/r3xb}\end{tikzpicture}
    }
    
\newcommand{\rIregions}{
    \begin{tikzpicture}[scale=0.4,
                baseline={([yshift=-\the\dimexpr\fontdimen22\textfont2\relax]
                    current bounding box.center)},
        ]
        \clip (-0.4,-1.2) rectangle + (3.6,2.4);
        \draw[thick, orange, out=90, in=90] (0,-1) to (3,0);
\draw[thick, orange, out=270, in=315] (3,0) to (1.1,-0.1);
\draw[thick, orange, out=135] (0.6,0.4) to (0,1);

\node[] at (-0.2,0.2) {\tiny \color{red} $A$};
\node[] at (1,1) {\tiny \color{red} $B$};
\node[] at (2,0) {\tiny \color{red} $C$};\end{tikzpicture}
    }
    
\newcommand{\rIaregions}{
    \begin{tikzpicture}[scale=0.4,
                baseline={([yshift=-\the\dimexpr\fontdimen22\textfont2\relax]
                    current bounding box.center)},
        ]
        \clip (-0.4,-1.2) rectangle + (3.6,2.4);
        \draw[thick, orange, out=45, in=270] (0,-1) to (0.5,0);
\draw[thick, orange, out=90, in=315] (0.5,0) to (0,1);
\draw[thick, orange, out=90, in=90] (1,0) to (3,0);
\draw[thick, orange, out=270, in=270] (3,0) to (1,0);

\node[] at (-0.2,0) {\tiny \color{red} $A$};
\node[] at (0.8,1) {\tiny \color{red} $B$};
\node[] at (2,0) {\tiny \color{red} $C$};\end{tikzpicture}
    }
    
\newcommand{\rIbregions}{
    \begin{tikzpicture}[scale=0.4,
                baseline={([yshift=-\the\dimexpr\fontdimen22\textfont2\relax]
                    current bounding box.center)},
        ]
        \clip (-0.4,-1.2) rectangle + (3.6,2.4);
        \draw[thick, orange, out=90, in=135] (0,-1) to (1, -0.4);
\draw[thick, orange, out=315, in=270] (1, -0.4) to (3, 0);
\draw[thick, orange, out=90, in=45] (3,0) to (1, 0.4);
\draw[thick, orange, out=225, in=270] (1, 0.4) to (0,1);

\node[] at (-0.2,0) {\tiny \color{red} $A$};
\node[] at (0.7,1) {\tiny \color{red} $B$};
\node[] at (2,0) {\tiny \color{red} $C$};\end{tikzpicture}
    }
    
\newcommand{\rIIregions}{
    \begin{tikzpicture}[scale=0.4,
                baseline={([yshift=-\the\dimexpr\fontdimen22\textfont2\relax]
                    current bounding box.center)},
        ]
        \clip (-0.4,-1.4) rectangle + (5.8,2.8);
        \draw[thick, orange, out=90, in=90] (0,-1) to (5,-1);
\draw[thick, orange, out=135, in=-45] (0.6,0.4) to (0,1);
\draw[thick, orange, out=225, in=315] (3.9,-0.1) to (1.1,-0.1);
\draw[thick, orange, out=45, in=225] (4.4,0.4) to (5,1);

\node[] at (-0.2,0.2) {\tiny \color{red} $A$};
\node[] at (2.5,1) {\tiny \color{red} $B$};
\node[] at (5.2,0.2) {\tiny \color{red} $C$};
\node[] at (2.5,-1.1) {\tiny \color{red} $D$};
\node[] at (2.5,-0.1) {\tiny \color{red} $E$};\end{tikzpicture}
    }
    
\newcommand{\rIIabProofOne}{
    \begin{tikzpicture}[scale=0.3,
                baseline={([yshift=-\the\dimexpr\fontdimen22\textfont2\relax]
                    current bounding box.center)},
        ]
        \clip (-1.2,-1.2) rectangle + (6.4,2.4);
        \input{diagrams/r2ab-proof1}\end{tikzpicture}
    }
    
\newcommand{\rIIbaProofOne}{
    \begin{tikzpicture}[scale=0.3,
                baseline={([yshift=-\the\dimexpr\fontdimen22\textfont2\relax]
                    current bounding box.center)},
        ]
        \clip (-1.2,-1.2) rectangle + (6.4,2.4);
        \input{diagrams/r2ba-proof1}\end{tikzpicture}
    }
    
\newcommand{\rIIabProofTwo}{
    \begin{tikzpicture}[scale=0.3,
                baseline={([yshift=-\the\dimexpr\fontdimen22\textfont2\relax]
                    current bounding box.center)},
        ]
        \clip (-0.2,-2.2) rectangle + (4.4,4.4);
        \input{diagrams/r2ab-proof2}\end{tikzpicture}
    }
    
\newcommand{\rIIbaProofTwo}{
    \begin{tikzpicture}[scale=0.3,
                baseline={([yshift=-\the\dimexpr\fontdimen22\textfont2\relax]
                    current bounding box.center)},
        ]
        \clip (-0.2,-2.2) rectangle + (4.4,4.4);
        \input{diagrams/r2ba-proof2}\end{tikzpicture}
    }
    
\newcommand{\pictureHangingPuzzle}{
    \begin{tikzpicture}[scale=0.5,
                baseline={([yshift=-\the\dimexpr\fontdimen22\textfont2\relax]
                    current bounding box.center)},
        ]
        \clip (-2.2,-2.2) rectangle + (5.4,4.4);
        \input{diagrams/pictureHangingPuzzle}\end{tikzpicture}
    }

\usepackage[style = numeric]{biblatex}
\bibliography{refs.bib}

\usepackage{hyperref}
\hypersetup{
    colorlinks = true,
    linkcolor = blue
}
\usepackage[capitalise,noabbrev,nameinlink]{cleveref}

\usepackage{multirow}

\usepackage{thmtools, thm-restate}

\theoremstyle{definition}
\declaretheorem[name=Definition,style=definition,numberwithin=section]{definition}
\declaretheorem[name=Example,style=definition,numberlike=definition]{example}

\theoremstyle{plain}
\declaretheorem[name=Lemma,style=plain,numberlike=definition]{lemma}
\declaretheorem[name=Theorem,style=plain,numberlike=definition]{theorem}
\declaretheorem[name=Corollary,style=plain,numberlike=definition]{corollary}
\declaretheorem[name=Proposition,style=plain,numberlike=definition]{proposition}

\theoremstyle{remark}
\declaretheorem[name=Remark,style=remark,numberlike=definition]{remark}

\newcommand*{\defeq}{\mathrel{\vcenter{\baselineskip0.5ex \lineskiplimit0pt
\hbox{\scriptsize.}\hbox{\scriptsize.}}}%
=}

\title{Khovanov homology for links in thickened multipunctured disks}
\author[Z. Winkeler]{Zachary Winkeler}
\date{}
\address {Department of Mathematics, Dartmouth College\\ Hanover, NH 03755}
\email {zachary.j.winkeler.gr@dartmouth.edu}
\urladdr{\href{http://math.dartmouth.edu/~zwinkeler}{http://math.dartmouth.edu/~zwinkeler}}

\DeclareMathOperator{\g}{g}
\DeclareMathOperator{\APS}{APS}
\DeclareMathOperator{\writhe}{wr}
\newcommand{\wt}[1]{\widetilde{#1}}
\DeclareMathOperator{\MKh}{MKh}
\newcommand{\join}{\vee}
\newcommand{\meet}{\wedge}
\DeclareMathOperator{\lowersets}{Lower}
\DeclareMathOperator{\subsets}{Sub}
\DeclareMathOperator{\submodules}{Sub}
\newcommand{\bigjoin}{\bigvee}

\begin{document}

\maketitle

\begin{abstract}
    We define a variant of Khovanov homology for links in thickened disks with multiple punctures. This theory is distinct from the one previously defined by Asaeda, Przytycki, and Sikora, but is related to it by a spectral sequence. Additionally, we show that there are spectral sequences induced by embeddings of thickened surfaces, which recover the spectral sequence from annular Khovanov homology to Khovanov homology as a special case.
\end{abstract}


\section{Introduction}

In \cite{khovanov2000}, Khovanov introduced a bi-graded homology theory $\Kh(L)$ for links $L$ in $S^3$, categoryifing the Jones polynomial.  Later in \cite{asaeda2004}, Asaeda, Przytycki, and Sikora (APS) defined a triply-graded homology theory $\APS(L)$ for band links in any $\I$-bundle over a surface, categorifying the Kauffman bracket skein module. The latter homology theory can be thought of as a generalization of the former, as it reduces to Khovanov's homology (after some grading adjustments) for links in the trivial $\I$-bundle over a disk (identified with $\R^3$). On the decategorified level, the Kauffman bracket skein module of $S^3$ reduces to the Jones polynomial (again, after a suitable ``renormalization'').

The solid torus $S^1 \cross D^2$ has the structure of an $\I$-bundle when thought of as a thickened annulus $A \cross \I$; in \cite{roberts2013}, it was shown that the specialization of APS homology for such links (sometimes called annular Khovanov homology, denoted $\AKh(L)$) can be obtained via a filtration on the Khovanov complex. This theory, and annular links in general, is interesting for a variety of reasons. It has given rise to invariants of braids \cite{grigsby2016} and braid conjugacy classes \cite{hubbard2016}, and has been shown to distinguish braids from other tangles \cite{grigsby2014}. Annnular Khovanov homology has a natural $\mathfrak{sl}_2$ action \cite{grigsby2017}, which generalizes to an $\mathfrak{sl}_n$ action on Khovanov-Rozansky homology \cite{queffelec2015}. One can define a Lee-type deformation of $\AKh(L)$ as well \cite{grigsby2016}.

In this paper, we use a generalized notion of filtration over an arbitrary poset. This construction has been studied before, for example in \cite{matschke2013} and \cite{guidolin2021}. We use these generalized filtrations to define a homology theory $\MKh(L)$ for links in thickened disks with multiple punctures. Specifically, given a link diagram $D$ for a link $L$ in a thickened $n$-punctured disk $\Sigma \cross \I$, we construct a $\Z^n$-filtration on the Khovanov complex $\CKh(D)$. This filtration essentially records the annular grading of a generator with respect to each of the $n$ punctures. We then define $\MKh(L)$ to be the homology of the associated graded complex. This homology is naturally triply-graded, as it inherits the homological and quantum gradings from Khovanov homology, as well as the multipunctured grading used to construct the filtration. Given a homological grading $i \in \Z$, a quantum grading $j \in \Z$, and a multipunctured grading $v \in \Z^n$, we denote the corresponding graded piece $\MKh^{i,j,v}(L)$. Therefore, the whole homology decomposes as:

\begin{equation*}
    \MKh(L) = \bigoplus_{i,j,v} \MKh^{i,j,v}(L)
\end{equation*}

In \cref{sec:invariance}, we prove that this is in fact an invariant:

\begin{theorem}
\label{thm:mkh-invariance}
    Let $L \subset \Sigma \cross \I$ be a link in a thickened $n$-punctured disk, and let $\field$ be a commutative ring. Then $\MKh(L;\field)$ is an invariant of such links, as a triply-graded $\field$-module. In fact, for any diagram $D$ of $L$, we get that the filtered chain homotopy type of $\CKh(D;\field)$ is an invariant of $L$.
\end{theorem}

While APS homology is already defined for band links in thickened surfaces, our theory differs in that it is defined for unframed links, and is graded by $H_1(\Sigma; \Z)$, the first integral homology group of the surface instead of by $\Z C(\Sigma)$, the free abelian group on homotopy classes of simple closed curves. Additionally, since $\MKh(L)$ is computed from a filtered chain complex, we have access to more filtration-related structure.

Despite their differences, we still have a relationship between $\APS(L)$ and $\MKh(L)$. In order to compare the two, we define a homology $\wt{\APS}(L)$ for links in thickened $n$-punctured disks, which is related to $\APS(L)$ by a change in gradings, and is an invariant of (unframed) links. This invariant is also a triply-graded module; here, $i \in \Z$ is the homological grading, $j \in \Z$ is the quantum grading, and $v\in \Z C(\Sigma)$ is the negative of the $\Psi$-grading from \cite{asaeda2004} (analogous to our multipunctured grading).

\begin{equation*}
    \wt{\APS}(L) = \bigoplus \wt{\APS}^{i,j,v}(L)
\end{equation*}

Define the graded Euler characteristics of $\wt{\APS}(L)$ and $\MKh(L)$ as

\begin{align*}
    \chi(\wt{\APS}(L)) &\defeq \sum_{i,j,v} (-1)^i q^j x^v \rk (\wt{\APS}^{i,j,v}(L)) \\
    \chi(\MKh(L)) &\defeq \sum_{i,j,v} (-1)^i q^j y^v \rk (H^{i,j,v}(L))
\end{align*}

Here, $x^v$ and $y^v$ are shorthand notations for certain products of variables $x_c$ and $y_c$, one variable for each basis element of $\Z C(\Sigma)$ and $H_1(\Sigma; \Z)$, respectively. See \cref{sec:aps-relationship} for details.

We then get that these two Laurent polynomials are related by a Hurewicz-like map $h: \Z C(\Sigma) \to H_1(\Sigma; \Z)$.

\begin{theorem}
\label{thm:euler-characteristic}
The map $h$ induces a ring map $\chi_h: \Z[q^{\pm 1},x^{\pm 1}] \to \Z[q^{\pm 1},y^{\pm 1}]$, which sends $\chi(\wt{\APS}(L))$ to $\chi(MKh(L))$.
\end{theorem}

We then construct a spectral sequence from $\wt{\APS}(L)$ to $\MKh(L)$. While our previous constructions use filtrations and gradings over $\Z^n$, we first need to ``flatten'' these to $\Z$-filtrations and $\Z$-gradings to make use of the conventional machinery of spectral sequences associated to filtered chain complexes.

\begin{theorem}
\label{thm:aps-spectral-sequence}
Let $\Sigma$ be a disk with $n$ punctures, and let $L$ be a link in $\Sigma \cross \I$. Then there is a spectral sequence with $E_1$ page isomorphic to $\wt{\APS}(L)$ that converges to $\MKh(L)$ (with flattened gradings).
\end{theorem}

On the other hand, our construction is in some ways like a generalization of annular Khovanov homology\footnote{Which therefore makes our theory a generalization of a specialization of a generalization of Khovanov homology.}. We identify $\AKh(L)$ as a special case of $\MKh(L)$, and construct a spectral sequence from the latter to the former:

\begin{theorem}
\label{thm:akh-spectral-sequence}
Let $\Sigma$ be a disk with $n > 1$ punctures, and let $L$ be a link in $\Sigma \cross \I$. Choose a puncture $p$, and let $A = \D^2 \setminus\set{p}$ be the annulus obtained from $\Sigma$ by filling in every puncture \textit{except $p$}. Then there is a spectral sequence with $E_1$ page isomorphic to $\MKh(L \subset \Sigma \cross \I)$ (with flattened grading) that converges to $\AKh(L \subset A\cross \I)$.
\end{theorem}

We also recover the well-known spectral sequence from $\AKh(L)$ to $\Kh(L)$, as well as a spectral sequence from $\MKh(L)$ to $\Kh(L)$:

\begin{theorem}
\label{thm:mkh-spectral-sequence}
For any link $L$ in an $n$-punctured disk $\Sigma$, there is a spectral sequence with $E_1$ page isomorphic to $\MKh(L \subset \Sigma \cross \I)$ (with flattened grading) that converges to $\Kh(L)$.
\end{theorem}

\cref{thm:akh-spectral-sequence} and \cref{thm:mkh-spectral-sequence} are special cases of a more general relationship between links that are related by thickened surface embeddings.

\begin{theorem}
\label{thm:spectral-sequence}
Let $\Sigma$ be a disk with $n$ punctures, and let $\Sigma'$ be a disk with a subset of those $n$ punctures. Let $L \subset \Sigma \subset \Sigma'$ be a link. Then there is a spectral sequence with $E_1$ page isomorphic to $\MKh(L \subset \Sigma \cross \I)$ that converges to $\MKh(L \subset \Sigma' \cross \I)$ (with flattened gradings).
\end{theorem}

In \cite{queffelec2018}, Queffelec and Wedrich define a functor from links in thickened surfaces to a category of foams which allows one to recover $\Kh(L)$, $\AKh(L)$, and $\APS(L)$, among other link invariants. They also exhibit similar spectral sequences corresponding to surface embeddings. It seems likely that $\MKh(L)$ has some relationship with this construction.

In another direction, $\MKh(L)$ may be useful for extracting information about braids from $\AKh(L)$. Stabilization and destabilization of braids can't be realized by ambient isotopy in $A \cross \I$, but they can be viewed as a sequence of adding and removing punctures to a diagram. These moves would correspond to spectral sequences between various values of $\MKh(L)$;

Since we started out trying to generalize annular Khovanov homology, it's natural to ask if $\MKh(L)$ is any better at distinguishing knots in thickened multipunctured disks. It turns out that it is, i.e.\ we can find knots $L,L' \subset \Sigma \cross \I$ that have isomorphic $\AKh(L \subset A \cross \I)$ around each puncture individually, but different $\MKh(L \subset \Sigma \cross \I)$. We give one such example in \cref{sec:further-examples}.

\subsection*{Organization}

In \cref{sec:background}, we establish our notation and conventions while giving some background information on Khovanov homology and its annular version. In \cref{sec:poset-filtrations}, we discuss some facts about filtrations by posets. We define our main invariant $\MKh(L)$ in \cref{sec:mkh}, then discuss the natural spectral sequences it comes equipped with in \cref{sec:spectral-sequences}. In \cref{sec:relationships}, we relate $\MKh(L)$ with Khovanov homology, annular Khovanov homology, and APS homology. In \cref{sec:invariance}, we prove that $\MKh(L)$ is in fact a link invariant. Finally, we give an example to demonstrate how $\MKh(L)$ can be non-trivial even when $\AKh(L)$ is trivial in \cref{sec:further-examples}.

\subsection*{Acknowledgements}

We thank Eli Grigsby for helpful conversation, as well as Ina Petkova for suggesting the project and comments on the draft of this paper. We also would like to thank Robert Lipshitz for helpful advice while revising this paper.


\section{Background}
\label{sec:background}

In this section, we will review some of the prerequisite topics for our main theorems. We start by introducing our conventions for knots and links. Next, we recall the definition of Khovanov homology and work through a simple example. Lastly, we establish some notation we'll use when discussing filtrations and review annular Khovanov homology in anticipation of our more general theory in \cref{sec:mkh}.

\subsection{Links}
For the purposes of this paper, an $n$-component \textbf{link} is a smooth embedding of $\coproduct_n S^1$ into $S^3$, considered up to ambient isotopy. We will represent links via \textbf{diagrams}, which are 4-valent graphs with vertices labeled by over/under crossing data. We think of diagrams as projections of links onto $\R^2$ with no triple points or tangencies. All links are assumed to be \textbf{oriented}, which means that we have chosen one of the two possible orientations for each of the link components. Given an oriented link diagram, we can define a \textbf{positive crossing} as a part of a diagram that locally looks like $\positiveCrossing$, and a \textbf{negative crossing} as a part of a diagram that locally looks like $\negativeCrossing$.

\subsection{Khovanov homology}
\label{sec:khovanov}

First defined in \cite{khovanov2000}, Khovanov homology is an invariant of links in $S^3$. We will give a short description of its definition here; for a more detailed description, see the original paper or Bar-Natan's paper \cite{bar-natan2002}.

We will use two gradings to define this invariant. The first grading is the homological grading $\g^h$; following Bar-Natan's convention, shifts in this grading will be denoted with brackets, i.e.\ $M[i]_n = M_{n-i}$. The second grading is the quantum grading $\g^q$; shifts in this grading will be denoted with curly braces, i.e.\ $M\{i\}_n = M_{n-i}$.

Let $L$ be a link in $\R^3 \subset S^3$. Choose a projection $\R^3 \to \R^2$ so that the image of $L$ is a 4-valent graph $D$ with crossing information at each vertex, i.e.\ a link diagram. Choose an ordering $c_1, \dots, c_n$ of the crossings of $D$. Given a crossing $\crossing$, define its 0-resolution to be $\zeroRes$ and its 1-resolution to be $\oneRes$. If $S \in \set{0,1}^n$ is any element, then define $D(S)$ to be the collection of closed circles in $\R^2$ obtained from $D$ by replacing each crossing $c_i$ with its $S_i$-resolution.

Fix a ring $\field$ (usually a field or $\Z$), and let $V=\field\generatedby{v_+,v_-}$ be the free $\field$-module generated by two elements $v_+$ and $v_-$. We define the quantum gradings of $v_+$ and $v_-$ to be $1$ and $-1$, respectively. To a fully-resolved diagram $D(S)$ with $k$ components, we associate a module $\CKh(D(S)) = V^{\tensor k}$. We think of a generator $x\in V^{\tensor k}$ as representing a labeling of each circle $C_i \in D(S)$ with a sign $s_i \in \set{+,-}$. We call such a labeling $\set{(C_i, s_i)}_i$ an \textbf{enhanced (Kauffman) state}. Let $n_+$ and $n_-$ be the number of positive crossings and negative crossings in $D$, respectively, and let $\writhe(D) = n_+-n_-$ be the \textbf{writhe} of $D$. To the original diagram $D$, we associate a module:
\[
    \CKh(D) \defeq \bigdirectsum_{S \in \set{0,1}^n} \CKh(D(S))[|S|-n_-]\{|S|-n_-+\writhe(D)\}
\]

When we care about the gradings, we will write $\CKh^{h,q}(D)$ to denote the graded piece of $\CKh(D)$ with homological grading $h$ and quantum grading $q$.

We can give $\CKh(D)$ the structure of a chain complex by defining a differential $\diff: \CKh(D) \to \CKh(D)$ on it. Let $S \in \set{0,1}^n$, and let $x \in \CKh(D(S))$ be a generator. Let $S' \in \set{0,1}^n$ be obtained from $S$ by changing a single 0 to a 1. If $D(S)$ has $k$ components, then $D(S')$ has either $k-1$ or $k+1$ components; we call the former a \textbf{merge} and the latter a \textbf{split}. If $S$ and $S'$ are related by a merge, then we can define a map $m: \CKh(D(S)) \iso V^{\tensor k} \to V^{\tensor k-1} \iso \CKh(D(S'))$. On the two tensor factors that are being merged, $m$ is defined as
\[
    m: V \tensor V \to V = 
        \piecewise{
            v_+ \tensor v_+ \mapsto v_+ \\
            v_+ \tensor v_- \mapsto v_- \\
            v_- \tensor v_+ \mapsto v_- \\
            v_- \tensor v_- \mapsto 0
        }
\]
We extend $m$ to all of $\CKh(D(S))$ by setting it to be the identity on all the other tensor factors. Similarly, if $S$ and $S'$ are related by a split, we define a map $\Delta: \CKh(D(S)) \iso V^{\tensor k} \to V^{\tensor k+1} \iso \CKh(D(S'))$. On the tensor factor that is being split, $\Delta$ is defined as
\[
    \Delta: V \to V \tensor V = 
        \piecewise{
            v_+ \mapsto v_+ \tensor v_- + v_- \tensor v_+ \\
            v_- \mapsto v_- \tensor v_- \\
        }
\]
We extend $\Delta$ to all of $\CKh(D(S))$ be setting it to be the identity on all the other tensor factors. Let $\diff_{S,S'}: \CKh(D(S)) \to \CKh(D(S'))$ denote the map $m$ if $S$ and $S'$ are related by a merge, $\Delta$ if $S$ and $S'$ are related by a split, and $0$ otherwise. We will implicitly extend the domain and codomain of $\diff_{S,S'}$ by zero when we define the Khovanov differential $\diff: \CKh(D) \to \CKh(D)$ as
\[
    \diff \defeq \sum_{S, S'\in\set{0,1}^n} \epsilon(S,S') \diff_{S,S'}
\]
Here $\epsilon(S,S') \in \set{1,-1}$ is some sign-assignment function to ensure that $\diff^2=0$. There are many possible functions to choose for $\epsilon$; one that works here is to let
\[
    \epsilon(S,S') \defeq (-1)^{\sum_{i=1}^{j-1} S_i}
\]
where $j$ is the first index at which $S_i \not= S'_i$. The Khovanov differential has degree $1$ with respect to the homological grading, and $0$ with respect to the quantum grading.

Finally, we define the \textbf{Khovanov homology} of $L$ (with coefficients in $\field$) as the homology of this chain complex\footnote{Since the differential has degree $1$, it would perhaps be more appropriate to refer to $\CKh(D)$ as a \textit{cochain} complex, and thus we would be taking its \textit{cohomology}. This justifies our use of superscripts, but we will continue to use words like ``chain complex'' and ``homology'', as this seems to be the convention.}:
\begin{equation*}
    \Kh(L;\field) \defeq H^*(\CKh(D),\diff)
\end{equation*}

We will often omit writing the coefficient ring if it is clear from context. One can prove that $\Kh$ is an invariant of the link $L$, and therefore does not depend on the choice of diagram $D$, ordering of the crossings $c_1,\dots,c_n$, or sign assignment $\epsilon$.

\begin{example}
Let $L=\hopfLink$ be the Hopf link. Orient both components counter-clockwise. Choose the base ring $\field=\Q$, so $V=\Q^2$. We use the cube of resolutions to form the chain complex $\CKh(L)$:
\cd{
\hopfLinkab \ar[r] \& \hopfLinkbb \&
V \ar[r,"\Delta_2"] \& V^{\tensor 2} \\
\hopfLinkaa \ar[u] \ar[r] \& \hopfLinkba \ar[u] \&
V^{\tensor 2} \ar[u,"m_1"] \ar[r,"m_2"] \& V \ar[u,"\Delta_1"]
}
We can write the above in a ``flatter'' version:
\cd[row sep=large, column sep=large]{
V^{\tensor 2} \ar[r,"\spmat{m_1 \\ m_2}"] \& V \directsum V \ar[r,"\spmat{-\Delta_2 & \Delta_1}"] \& V^{\tensor 2}
}
Replacing $V$, $m$, and $\Delta$ by their definitions gives:
\cd[row sep=huge, column sep=huge]{
\Q^4 \ar[r,"\spmat{1 & 0 & 0 & 0 \\ 0 & 1 & 1 & 0 \\ 1 & 0 & 0 & 0 \\ 0 & 1 & 1 & 0}"] \& \Q^4 \ar[r,"\spmat{0 & 0 & 0 & 0 \\ -1 & 0 & 1 & 0 \\ -1 & 0 & 1 & 0 \\ 0 & -1 & 0 & 1}"] \& \Q^4
}
Note that we had to choose an ordered basis of each $V^{\tensor n}$ to write the matrices above. The convention we will use in this paper is that tensor factors are sorted left-to-right by the leftmost points in the loops representing them. Additionally, the signs labelling each loop start at all $+$ signs and count up in binary (e.g.\ $++$,$+-$,$-+$,$--$). For example, the 3rd column of the left matrix above says that
\eq{
    \diff(\underset{v_- \tensor v_+}{\hopfLinkaa}) &= \underset{v_-}{\hopfLinkab} + \underset{v_-}{\hopfLinkba}
}

Since our diagram contains no positive crossings and two negative crossings, we see that $n_+=0$, $n_-=-2$, and $\writhe(D) = -2$, which tells us how to shift the gradings. To find the Khovanov homology of the Hopf link, we take the homology of the Khovanov complex to find that:
\eq{
    \Kh^{-2}(L) &\iso \Q\{-4\} \directsum \Q\{-6\} \\
    \Kh^{-1}(L) &\iso 0 \\
    \Kh^0(L) &\iso \Q \directsum \Q\{-2\}
}
\end{example}


\subsection{$\Z$-filtrations}
\label{sec:z-filtrations}

In the process of recalling annular Khovanov homology in \cref{sec:akh}, we will need to construct a filtration on the Khovanov complex. Our homology in \cref{sec:mkh} will also be defined using a more general notion of $\Z^n$-filtration. We review the theory of filtrations by $\Z$ (which we may call $\Z$-filtrations) below.

The usual notion of a \textbf{(bounded, ascending) filtration} $\filt$ of an $R$-module $M$ is a sequence of submodules:
\[
    \ldots \subseteq \filt_{-1} M \subseteq \filt_0 M \subseteq \filt_1 M \subseteq \ldots
\]
such that $\filt_i M = 0$ and $\filt_j M = M$ for some $i,j \in \Z$. Additionally, we will require that $\filt_s M = 0$ and $\filt_t M = M$ for some $s, t \in \Z$.

A module $M$ equipped with a filtration $\filt$ is a \textbf{filtered module} $(M,\filt)$. We will often overload notation by referring to filtered modules by their underlying modules. If two modules $M$ and $N$ are filtered, then a map $f: M \to N$ is \textbf{filtered} if $f(\filt_a M) \subseteq \filt_a N$ for all $a \in \Z$. A \textbf{filtered isomorphism} is a filtered map with a filtered inverse.

Given a filtered module $M$ and an injective map of modules $f: N \inj M$, define the \textbf{induced filtration} on $N$ to be the filtration $\filt_a'N = f^{-1}(\filt_a M)$ for $a \in \Z$. Oftentimes, $N \inj M$ will be thought of as an inclusion of a submodule. Similarly, given a filtered module $M$ and a surjective map of modules $f: M \surj N$, define the \textbf{quotient filtration} on $N$ to be the filtration $\filt_a'N = f(\filt_b M)$ for $b \in \Z$. Oftentimes, $M \surj N$ will be thought of as a quotient by a submodule.

\subsubsection{Associated graded objects}

Given a filtered module $M$, we define the \textbf{associated graded module} $\gr(M)$ as follows:
\eq{
    \gr(M) &\defeq \bigdirectsum_{a\in\Z} \gr[a](M) \\
    \gr[a](M) &\defeq \frac{\filt_a M}{\filt_{a-1} M}
}
For a vector space $V$, it is easy to see that $V \iso \gr(V)$ as modules; this is not true for modules over general rings. Because it will be useful in the proof of \cref{thm:mkh-invariance} later, we define a \textbf{filtered projective} module $M$ to be a filtered module such that $\gr[a](M)$ is projective for all $a \in \Z$.

\subsubsection{$\Z$-filtered complexes}

A \textbf{filtration of a chain complex} $(C, \diff)$ is a filtration of $C$ that respects $\diff$ in the sense that $\diff (\filt_a C) \subseteq \filt_a C$ for $a \in \Z$. A \textbf{filtered chain complex} is a chain complex $(C, \diff)$ equipped with a filtration $\filt$; we will often overload notation by simply referring to $(C,\diff,\filt)$ as $C$. A \textbf{filtered chain map} is a map of chain complexes that also respects the filtration, i.e.\ a map of modules $f: C \to D$ that commutes with $\diff$ and ``commutes'' with $\filt$. Similarly, the \textbf{associated graded complex} $\gr(C,\diff,\filt) = (\gr(C), \gr(\diff))$ has underlying module the associated graded module and differential induced by the quotient operation. Even for chain complexes of vector spaces, we no longer have that $C \iso \gr(C)$ as chain complexes.

A \textbf{quasi-isomorphism} is a map $f: C \to D$ of chain complexes that induces an isomorphism $f_*: H_*(C) \to H_*(D)$. A \textbf{filtered quasi-isomorphism} is a filtered chain map $f: C \to D$ that induces quasi-isomorphisms $\gr[a](f): \gr[a](C) \to \gr[a](D)$ for all $a \in \Z$. Note that this is a stronger condition than just asking that $f$ is a filtered map and also a quasi-isomorphism. The class of filtered quasi-isomorphisms includes filtered isomorphisms. It will be relevant later that the composition of two filtered quasi-isomorphisms is again a filtered quasi-isomorphism. 

A filtered map $f: C \to D$ is said to be \textbf{strict} if $f(\filt_a C) = f(C) \intersect \filt_a D$ (not just $\subseteq$) for all $a \in \Z$. Essentially by definition, inclusions of filtered submodules and projections onto filtered quotient modules are always strict. The following fact about strict maps of filtered complexes will be useful later:

\begin{lemma}\label{lemma:strict-map-exact-sequence}
If $f: C \to D$ is a strict map of filtered complexes, then the following sequence is exact:
\begin{center}\begin{tikzcd}
0 \ar[r] & \gr(\ker f) \ar[r] & \gr(C) \ar[r, "\gr(f)"] & \gr(D) \ar[r] & \gr(\coker f) \ar[r] & 0
\end{tikzcd}\end{center}
Specifically, if $f$ is injective (surjective), then $\ker f$ (respectively, $\coker f$) is zero, so the above gives us a short exact sequence of chain complexes.
\end{lemma}

A proof of the above lemma can be found at \cite[\href{https://stacks.math.columbia.edu/tag/0120}{Section 0120}]{stacks-project}.

\subsubsection{$\Z$-filtered vector spaces}

Since we are going to be working over a field for most of this paper, we can make some convenient simplifications. A filtered vector space $V$ is isomorphic to its associated graded vector space $\gr(V)$. Therefore, a $\Z$-filtration $\filt$ of a vector space $V$ is equivalent to a $\Z$-grading $\g$ of $V$. Under this correspondence, filtered isomorphisms and graded isomorphisms coincide. The following lemmas explain how these gradings interact with the induced and quotient filtrations described earlier.

\begin{lemma}\label{induced-grading}
Let $V$ be a filtered vector space with grading $\g$, and let $f: W \to V$ be an injective linear map. Give $W$ the induced filtration; the associated grading on a homogeneous element $w \in W$ is then given by $\g(f(w))$.
\end{lemma}

\begin{lemma}\label{quotient-grading}
Let $V$ be a filtered vector space with grading $\g$, and let $f: V \to W$ be a surjective linear map. Give $W$ the quotient filtration; the associated grading on a homogeneous element $w \in W$ is then given by the minimum of $\g(v)$ over all homogeneous elements $v\in f^{-1}(w)$.
\end{lemma}


\subsection{Annular Khovanov homology}
\label{sec:akh}

Now that we have filtrations available to us, we can equip the Khovanov complex with one to give a simple description of annular Khovanov homology.

Let $A=\D^2\setminus *$ be an annulus. If $D$ is a diagram for a link $L \subset A\times I$, then one can define more structure on the Khovanov complex $(\CKh(D),\diff)$. If $x=\set{(C_i, s_i)}_i$ is an enhanced state of $\CKh(D)$, let
\begin{equation*}
    k(x) \defeq \sum_{i=1}^j s_i\equivclass{C_i} \qquad \in H_1(A; \Z)
\end{equation*}
Since $H_1(A; \Z) \iso \Z$, $k$ actually specifies a grading on the underlying module $\CKh(L)$. In \cite{roberts2013}, Roberts proved that the filtration $\filt^A$ induced by $k$ on $(\CKh(L),\diff)$ is compatible with the Khovanov differential, and thus we can view $(\CKh(L),\diff,\filt^A)$ as a filtered chain complex. The \textbf{annular Khovanov homology} of $L$ is then defined to be
\begin{equation*}
    \AKh(L) \defeq H^*(\gr(\CKh(L), \diff, \filt^A))
\end{equation*}
That the above is actually an invariant of annular links was first proved by Asaeda, Przytycki, and Sikora in \cite{asaeda2004}, although not in those exact words.

\begin{example}
Let $L=\annularHopfLink$; here, the black dot represents the missing center of the annulus. We construct a similar cube of resolutions as before:
\cd{
\annularHopfLinkab \ar[r] \& \annularHopfLinkbb \&
V \ar[r,"\Delta_2"] \& V^{\tensor 2} \\
\annularHopfLinkaa \ar[u] \ar[r] \& \annularHopfLinkba \ar[u] \&
V^{\tensor 2} \ar[u,"m_1"] \ar[r,"m_2"] \& V \ar[u,"\Delta_1"]
}
As before, the differential can be written explicity as:
\cd[row sep=huge, column sep=huge]{
\Q^4 \ar[r,"\spmat{1 & 0 & 0 & 0 \\ 0 & 1 & 1 & 0 \\ 1 & 0 & 0 & 0 \\ 0 & 1 & 1 & 0}"] \& \Q^4 \ar[r,"\spmat{0 & 0 & 0 & 0 \\ -1 & 0 & 1 & 0 \\ -1 & 0 & 1 & 0 \\ 0 & -1 & 0 & 1}"] \& \Q^4
}
The $k$-grading of a generator represented by one of the diagrams $\annularHopfLinkaa$, $\annularHopfLinkab$, or $\annularHopfLinkba$ is 0. The $k$-grading of a generator represented by $\annularHopfLinkbb$ is $-2$ if both circles are marked with a $-$, $0$ if the circles have opposite signs, and $2$ if both circles are marked with a $+$. The differential on the above complex is filtered with respect to this grading, so to calculate annular Khovanov homology, we consider the associated graded complex.
\cd[row sep=huge, column sep=huge]{
\Q^4 \ar[r,"\spmat{1 & 0 & 0 & 0 \\ 0 & 1 & 1 & 0 \\ 1 & 0 & 0 & 0 \\ 0 & 1 & 1 & 0}"] \& \Q^4 \ar[r,"\spmat{0 & 0 & 0 & 0 \\ -1 & 0 & 1 & 0 \\ -1 & 0 & 1 & 0 \\ 0 & 0 & 0 & 0}"] \& \Q^4
}
One can think of the differential on this complex as the original Khovanov differential, but without the two maps
\eq{
\underset{v_-}{\annularHopfLinkab} \mapsto \underset{v_- \tensor v_-}{\annularHopfLinkbb}
\quad \text{ and } \quad 
\underset{v_-}{\annularHopfLinkba} \mapsto \underset{v_- \tensor v_-}{\annularHopfLinkbb}
}
since these lower the $k$-grading. This complex allows us to calculate that:
\eq{
    \AKh^{-2}(L) &\iso \Q\{-4\} \directsum \Q\{-6\} \\
    \AKh^{-1}(L) &\iso \Q\{-4\} \\
    \AKh^0(L) &\iso \Q \directsum \Q\{-2\} \directsum \Q\{-4\}
}
\end{example}


\section{Filtrations by Posets}
\label{sec:poset-filtrations}

In this section, we generalize the notion of a filtration from \cref{sec:z-filtrations}; instead of indexing subobjects over $\Z$, we allow filtrations by arbitrary posets $\poset$ subject to some conditions. This general notion of filtration will be used to define our main invariant in \cref{sec:mkh}.

\subsection{$\poset$-filtrations}

Recall that a \textbf{partially-ordered set} (usually abbreviated ``poset'') is a set $\poset$ equipped with a partial order $\le$, i.e. a relation that is reflexive, anti-symmetric, and transitive. A poset is \textbf{bounded} if there exists a least element $\bot \in \poset$ and a greatest element $\top \in \poset$. A map of posets is simply a function that respects the partial order. A \textbf{lattice} is a poset in which every pair of elements $a,b$ has a greatest lower bound $a \meet b$ (called their ``meet''), and least upper bound $a \join b$ (called their ``join''). A lattice is \textbf{distributive} if these operations distribute across one another. A \textbf{lattice homomorphism} $L \to L'$ is a map of posets that preserves meets and joins.

One way to get a lattice is to take the set $\submodules(M)$ of submodules of any $\field$-module $M$; given $A,B \subseteq M$, we define $A \le B$ iff $A \subseteq B$, $A \meet B \defeq A \intersect B$, and $A \join B \defeq A + B$. This lattice is not, in general, distributive.

\begin{definition}
For a poset $\poset$, a \textbf{$\poset$-filtration} of a $\field$-module $M$ is a map $\filt_{-}: \poset \to \submodules(M)$, i.e. a choice of submodule $\filt_a M$ for each $a \in \poset$ such that $\filt_a M \subseteq \filt_b M$ whenever $a \le b$.
\end{definition}

A filtration is \textbf{bounded} if there exists a finite interval $[s,t]\subseteq \poset$ such that the restriction of $\filt_{-}$ to $[s,t]$ is a map of bounded posets i.e.\ $\filt_s = 0$ and $\filt_t = M$. We will call a filtration \textbf{distributive} if the sublattice of $\submodules(M)$ generated by the image of $\filt_{-}$ is distributive. These properties ensure that the associated graded objects are suitably nice. For this reason, we will assume that all $\poset$-filtrations are bounded and distributive for the rest of the paper.

A \textbf{map of $\poset$-filtered modules} is a linear map $f: M \to N$ such that $f(\filt_a M) \subseteq \filt_a(N)$ for all $a \in \poset$. We will call a map of $\poset$-filtered modules \textbf{distributive} if the sublattice of $\submodules(N)$ generated by $f(\filt_{a}M)$ and $\filt_{a}N$ for all $a \in \poset$ is distributive. Again, we will assume that all maps of filtered modules are distributive for the rest of the paper.

Almost all of our constructions and definitions from \cref{sec:z-filtrations} can be generalized to posets, i.e.\ filtered isomorphisms, the induced and quotient filtration, and so on. The only notable exception is the associated graded object, which we will deal with shortly.

\subsubsection{Associated graded objects}

Given a $\poset$-filtered module $M$, we define the \textbf{associated graded module} $\gr(M)$ to be the module
\eq{
    \gr(M) &\defeq \bigdirectsum_{a \in \poset} \gr[a](M) \\
    \gr[a](M) &\defeq \frac{\filt_a M}{\sum\limits_{b < a} \filt_b M}
}

Here, the $\sum$ in the denominator refers to the sum of $\filt_b M$ as submodules of $\filt_a M$. One benefit of our definition of $\poset$-filtered module is that it ensures that the associated graded is not ``too big'':
\begin{theorem}\label{thm:rank-of-gr}
$\rank M = \rank \gr(M)$.
\end{theorem}

In proving this theorem, it will help to introduce a bit more notation. Let
\begin{equation*}
    \gr[\le a](M) \defeq \bigdirectsum_{b \le a} \gr[b](M)
\end{equation*}
It turns out that the rank of the above module satisfies some nice properties, which we will encode in the following definition. A \textbf{valuation} on a bounded lattice $L$ (see \cite{klain1997}) is a function $\mu: L \to \R$ such that
\begin{itemize}
    \item $\mu(a \join b) = \mu(a) + \mu(b) - \mu(a \meet b)$
    \item $\mu(\bot) = 0$
\end{itemize}
When $L$ is distributive, we can iterate the above conditions to obtain the \textit{inclusion-exclusion principle}:
\begin{equation*}
    \mu(a_1 \join a_2 \join \dots \join a_n) = \sum_i \mu(a_i) - \sum_{i<j} \mu(a_i \meet a_j) + \sum_{i < j < k} \mu(a_i \meet a_j \meet a_k) - \dots
\end{equation*}
Essentially, this tells us that, in a distributive lattice, the valuation of a join of elements is entirely determined by the valuations of the meets of the elements.
We will use this notion to prove the next lemma.

\begin{lemma}\label{lemma:inductive-hypothesis}
$\rank \gr[\le a](M) = \rank \filt_a(M)$.
\end{lemma}

\begin{proof}[Proof of \cref{lemma:inductive-hypothesis}]
First, we will view both sides of the equation as valuations on the lattice $L$ generated by the image of $\filt_{-}$. Let $u(a) \defeq \rank \gr[\le a](M)$, and let $v(a) \defeq \rank \filt_a M$. We can see that both $u$ and $v$ are valuations, since
\begin{align*}
    u(\bot)
    &= \rank \gr[\le \bot](M) \\
    &= \rank \gr[\bot](M) \\
    &= 0 \\
    v(\bot)
    &= \rank \filt_\bot M \\
    &= 0 \\
    u(a \join b)
    &= \rank \gr[\le a \join b](M) \\
    &= \rank \gr[\le a](M) + \rank \gr[\le b](M) - \rank \gr[\le a \meet b](M) \\
    &= u(a) + u(b) - u(a \meet b) \\
    v(a \join b)
    &= \rank \filt_{a \join b} M \\
    &= \rank(\filt_a M + \filt_b M) \\
    &= \rank \filt_a M + \rank \filt_b M - \rank(\filt_a M \intersect \filt_b M) \\
    &= \rank \filt_a M + \rank \filt_b M - \rank \filt_{a \meet b} M \\
    &= v(a) + v(b) - v(a \meet b)
\end{align*}

With this in mind, we would like to prove that $u$ and $v$ are actually the \textit{same} valuation on $\lowersets(\poset)$.

We will proceed by induction. Our base case is already done for us, as $u(\bot) = 0 = v(\bot)$. For the inductive case, fix $a \in L$ and assume that we have proven the claim for all $b \in L$ such that $b < a$. Then
\begin{align*}
    u(a)
    &= \rank \gr[\le a](M) \\
    &= \rank \bigdirectsum_{b \le a} \gr[b](M) \\
    &= \sum_{b \le a} \rank \gr[b](M) \\
    &= \rank \gr[a](M) + \sum_{b < a} \rank \gr[b](M) \\
    &= \rank \gr[a](M) + u\left(\bigjoin_{b < a} b\right) \\
    &= \rank \gr[a](M) + v\left(\bigjoin_{b < a} b\right) &\text{(by inclusion-exclusion)} \\
    &= \rank \gr[a](M) + \rank \sum_{b < a} \filt_b M \\
    &= \rank \filt_a M &\text{(by the definition of $\gr[a](M)$)} \\
    &= v(a) & &\qedhere
\end{align*}
\end{proof}

\begin{proof}[Proof of \cref{thm:rank-of-gr}]
This follows from \cref{lemma:inductive-hypothesis} with $a = \top \in \poset$.
\end{proof}

\subsubsection{$\poset$-filtered complexes}

As in the $\Z$-filtered case, a \textbf{$\poset$-filtration of a chain complex} $(C, \diff)$ is a $\poset$-filtration of $C$ that respects $\diff$ in the sense that $\diff (\filt_a C) \subseteq \filt_a C$ for $a \in \poset$. A \textbf{$\poset$-filtered chain map} is a map of chain complexes that also respects the filtration, i.e.\ a map of modules $f: C \to D$ that commutes with $\diff$ and ``commutes'' with $\filt$. Similarly, the \textbf{associated graded complex} $\gr(C,\diff,\filt) = (\gr(C), \gr(\diff))$ has underlying module the associated graded module and differential induced by the quotient operation. A \textbf{$\poset$-filtered quasi-isomorphism} is a filtered chain map $f: C \to D$ that induces quasi-isomorphisms $\gr[a](f): \gr[a](C) \to \gr[a](D)$ for all $a \in \poset$.

We define strictness for $\poset$-filtered maps the same way as we did for $\Z$-filtrations: $f: C \to D$ is strict iff $f(\filt_a C) = f(C) \intersect \filt_a D$ for all $a \in \poset$. Since it relies on our new definition of associated graded, we will generalize the proof of \cref{lemma:strict-map-exact-sequence} from \cite[\href{https://stacks.math.columbia.edu/tag/0120}{Section 0120}]{stacks-project} to the $\poset$-filtered case.

\begin{lemma}\label{lemma:p-filtered-short-exact-sequence}
Let $X$ be a $\poset$-filtered complex, and let $X \subseteq A$ be a filtered subcomplex. Then
\begin{center}\begin{tikzcd}
0 \ar[r] & \gr(X) \ar[r] & \gr(A) \ar[r] & \gr(A/X) \ar[r] & 0
\end{tikzcd}\end{center}
is exact.
\end{lemma}

\begin{proof}
First, we rewrite $\gr[a](X)$ as
\begin{align*}
    \gr[a](X) = \frac{\filt_a X}{\sum_{b < a} \filt_b X} = \frac{X \intersect \filt_a A}{\sum_{b < a} X \intersect \filt_b A} &= \frac{X \intersect \filt_a A}{X \intersect \sum_{b < a} \filt_b A} \\
    &\iso \frac{X \intersect \filt_a A + \sum_{b < a} \filt_b A}{\sum_{b < a} \filt_b A}
\end{align*}
to see that the induced map to
\begin{equation*}
\gr[a](A) = \frac{\filt_a A}{\sum_{b < a} \filt_b A}
\end{equation*}
is injective. Similarly, letting $\pi: A \to A/X$, we rewrite $\gr[a](A/X)$ as
\begin{align*}
    \gr[a](A/X) &= \frac{\filt_a A/X}{\sum_{b < a} \filt_b A/X} = \frac{\pi(\filt_a A)}{\sum_{b < a} \pi(\filt_b A)} = \frac{\pi(\filt_a A)}{\pi(\sum_{b < a} \filt_b A)} \\
    &= \frac{\frac{\filt_a A}{X \intersect \filt_a A}}{\frac{\sum_{b < a} \filt_b A}{X \intersect \sum_{b < a} \filt_b A}} = \frac{\frac{\filt_a A}{X \intersect \filt_a A}}{\frac{\sum_{b < a} \filt_b A}{X \intersect \filt_a A \intersect \sum_{b < a} \filt_b A}} \\
    &\iso \frac{\frac{\filt_a A}{X \intersect \filt_a A}}{\frac{X \intersect \filt_a A + \sum_{b < a} \filt_b A}{X \intersect \filt_a A}} \\
    &\iso \frac{\filt_a A}{X \intersect \filt_a A + \sum_{b < a} \filt_b A}
\end{align*}
to see that the induced map $\gr[a](A) \to \gr[a](A/X)$ is surjective. Thus, all we have left to do is check exactness in the middle. We have actually already done the work for this part, though; checking the numerator of $\gr[a](X)$ and the denominator of $\gr[a](A/X)$, we see that $\im(\gr[a](X) \to \gr[a](A))$ and $\ker(\gr[a](A) \to \gr[a](A/X))$ are equal, thus the sequence is exact.
\end{proof}

\begin{lemma}\label{lemma:p-filtered-strict-map-exact-sequence}
If $f: C \to D$ is a strict map of $\poset$-filtered complexes, then the following sequence is exact:
\begin{center}\begin{tikzcd}
0 \ar[r] & \gr(\ker f) \ar[r] & \gr(C) \ar[r, "\gr(f)"] & \gr(D) \ar[r] & \gr(\coker f) \ar[r] & 0
\end{tikzcd}\end{center}
\end{lemma}

\begin{proof}
First, we get that
\begin{center}\begin{tikzcd}
0 \ar[r] & \gr(\ker f) \ar[r] & \gr(C) \ar[r] & \gr(\coim(f)) \ar[r] & 0
\end{tikzcd}\end{center}
and
\begin{center}\begin{tikzcd}
0 \ar[r] & \gr(\im f) \ar[r] & \gr(D) \ar[r] & \gr(\coker(f)) \ar[r] & 0
\end{tikzcd}\end{center}
are exact as special cases of \cref{lemma:p-filtered-short-exact-sequence}. Here, $\coim(f) = C/\ker(f)$ is the coimage of $f$. We can stitch these two sequences together if we can show that $\gr(\coim(f)) \to \gr(\im(f))$ is an isomorphism.

To accomplish this, we will show that if $f$ is strict, then the map $\coim(f) \to \im(f)$ is a filtered isomorphism. This map is always an isomorphism of chain complexes, so it suffices to check that the map respects the filtrations. On $\coim(f)$, we have the quotient filtration from $C \surj \coim(f)$, and on $\im(f)$ we have the induced filtration from $\im(f) \inj D$. Since $f$ is strict, we get that these filtrations coincide, i.e.\ for all $a \in \poset$,
\begin{equation*}
f(\filt_a \coim(f)) = f(\filt_a C) = f(C) \intersect \filt_a D = \filt_a \im(f)
\end{equation*}

By functoriality, we then get that $\gr(\coim(f)) \to \gr(\im(f))$ is an isomorphism of graded complexes. Therefore, we get that
\begin{center}\begin{tikzcd}
0 \ar[r] & \gr(\ker f) \ar[r] & \gr(C) \ar[rr, "\gr(f)"] \ar[dr] &[-40pt] &[-40pt] \gr(D) \ar[r] & \gr(\coker f) \ar[r] & 0 \\
& & & \gr(\coim(f)) \iso \gr(\im(f)) \ar[ur]
\end{tikzcd}\end{center}
is exact.
\end{proof}

\subsubsection{$\poset$-filtered vector spaces}
\label{sec:filtered-vector-spaces}

$\poset$-filtered vector spaces have a simpler characterization. Let $\field$ be a field, and let $V$ be a $\field$-vector space. As we have previously mentioned, the image of $\filt_{-}$ in $\submodules(V)$ generates a finite distributive lattice. It turns out that such a lattice is isomorphic to a lattice of subsets, with operations given by set intersection and union. Specifically, $V$ decomposes as a direct sum of subspaces $V_a$ such that $\filt_a V = \bigdirectsum_{b \le a} V_a$ \cite{polishchuk2005}. Equivalently, there exists a basis $\beta$ of $V$ such that each $\filt_a V$ is the span of a subset of $\beta$. Therefore, we could instead specify a $\poset$-filtration of $V$ by fixing a basis $\beta$ of $V$ and giving a map from $\poset$ to $\subsets(\beta)$ (the poset of subsets of $\beta$). This is equivalent to specifying a grading of $V$ by $\poset$, which is the approach we will take in \cref{sec:mkh} and onward.


\section{Multipunctured Khovanov homology}
\label{sec:mkh}

In this section, we define the main construction of this paper. We start by defining a $\Z^n$-grading on (the underlying module of) the Khovanov complex of a link diagram in an $n$-punctured disk, then use it to build a homology theory in a way that is analogous to annular Khovanov homology.

Let $\I = [0,1]$ be the unit interval, let $\Sigma = \D^2\setminus\{p_1,\dots,p_n\}$ be an oriented disk with $n$ punctures, and let $L \subset \Sigma \cross \I$ be any link (considered up to ambient isotopy). Let $D \subset \Sigma$ denote the link diagram obtained by projecting $L$ onto $\Sigma$ (with small perturbations to remove triple points and tangencies if necessary), and let $D$ be the link diagram in $\R^2$ induced by the inclusion $\Sigma \inj \R^2$.

After picking an ordering of the crossings in $D$, we can construct the Khovanov complex $\CKh(D)$ as in \cref{sec:khovanov}. Let $\field = \Q$ be our coefficient ring. We would like to construct an $H_1(\Sigma; \Z)$-grading $\g^\Sigma$ on the module $\CKh(D)$ that ``remembers'' the extra structure of $D \subset \Sigma$; it suffices to specify such a grading on the generators. Let $x\in \CKh(D)$ be a generator represented by the enhanced state $\set{(C_i, s_i)}_i$; the $H_1(\Sigma;\Z)$-grading of $x$ is then defined to be
\begin{equation}\label{eq:grading}
    \g^\Sigma(x) \defeq \sum_{i=1}^j s_i \equivclass{C_i} \qquad \in H_1(\Sigma;\Z)
\end{equation}
where $\equivclass{C_i}$ is the homology class represented by $C_i$ in $\Sigma$.

Note that the differential on $\CKh(D)$ does not respect the $\g^\Sigma$. We claim that the differential on $\CKh(D)$ does, however, respect the filtration $\filt^{\Sigma}$ induced by $\g^\Sigma$, if we view $H_1(\Sigma;\Z)$ as a poset in a certain way. First, choose the basis of $H_1(\Sigma;\Z)$ represented by positively-oriented loops around each puncture. We use this basis to identify $H_1(\Sigma;\Z)$ with $\Z^n$. Then, give it the product partial order induced by the usual total order on $\Z$, i.e.
\begin{equation*}
    (i_1, i_2, \dots, i_n) \le (j_1, j_2, \dots, j_n) \iff (i_1 \le j_1) \land (i_2 \le j_2) \land \dots \land (i_n \le j_n)
\end{equation*}
 Our claim is thus:

\begin{lemma}\label{lemma:proof-filtered-differential}
The Khovanov differential $\diff$ respects the filtration $\filt^\Sigma$ on $\CKh(D)$.
\end{lemma}

\begin{proof}
We want to show that the differential and filtration on $(\CKh(D),\diff, \filt^\Sigma)$ are compatible, i.e.\ that $\diff (\filt^\Sigma_p \CKh(D)) \subseteq \filt^\Sigma_p \CKh(D)$. Thankfully, this proof comes almost for free from annular Khovanov homology. Choose a puncture $p_i$ and consider $D \subset \Sigma$ as a diagram in $\Sigma' = \D^2 \setminus \set{p_i}$; essentially, we are ``forgetting'' all but one of of the punctures. Let $\filt'$ be the filtration on $\CKh(D \subset \Sigma')$. $D \subset \Sigma'$ is an annular link diagram; in \cite{roberts2013}, it was shown that the Khovanov differential is compatible with a certain filtration on the Khovanov complex of an annular link. When the diagram $D \subset \Sigma'$ only has one puncture, the filtration induced by our grading $\g^\Sigma$ is equivalent to their filtration, so $\diff (\filt_p' \CKh(D)) \subseteq \filt_p' \CKh(D)$. Since this is true regardless of which puncture $p_i$ we pick, this implies that $\diff (\filt^\Sigma_p \CKh(D)) \subseteq \filt^\Sigma_p \CKh(D)$. Thus, the filtration on $\CKh(D \subset \Sigma)$ is compatible with the chain complex structure.
\end{proof}

Note that, since the complex $\CKh(D)$ is finitely-generated over $\Q$, the set of $\g^\Sigma$-gradings of generators of $\CKh(D)$ is finite, and in particular bounded. Therefore, there are elements $s, t \in \Z^n$ such that $s < \g^\Sigma(x) < t$ for all generators $x$ of $\CKh(D)$. This means that $\filt^\Sigma_s = 0$ and $\filt^\Sigma_t = \CKh(D)$, so our filtration is bounded. Additionally, since our filtration is induced by a grading, by \cref{sec:filtered-vector-spaces} it is distributive as well.

Now that we know that $\filt^\Sigma$ is a bounded, distributive filtration on the chain complex $\CKh(D)$, we can define our invariant.

\begin{definition}\label{def:invariant}
The link invariant $\MKh(L)$\footnote{We chose $\MKh$ for lack of a better abbreviation. The ``M'' stands for ``multipunctured'', not ``Mikhail''.} is defined to be the homology of the associated graded complex of $\CKh(D)$ with respect to $\filt^{\Sigma}$.
\begin{equation*}
    \MKh(L) \defeq H^*(\gr(\CKh(D),\diff, \filt^{\Sigma}))
\end{equation*}
\end{definition}
That the above is actually a well-defined link invariant is the content of \cref{thm:mkh-invariance}, which we prove in \cref{sec:invariance}.

\begin{example}

Let $L=\multiHopfLink$. Here, the black dots represent punctures. We again construct a cube of resolutions:

\cd{
\multiHopfLinkab \ar[r] \& \multiHopfLinkbb \&
V \ar[r,"\Delta_2"] \& V^{\tensor 2} \\
\multiHopfLinkaa \ar[u] \ar[r] \& \multiHopfLinkba \ar[u] \&
V^{\tensor 2} \ar[u,"m_1"] \ar[r,"m_2"] \& V \ar[u,"\Delta_1"]
}

The differential can be written explicitly as:
\cd[row sep=huge, column sep=huge]{
\Q^4 \ar[r,"\spmat{1 & 0 & 0 & 0 \\ 0 & 1 & 1 & 0 \\ 1 & 0 & 0 & 0 \\ 0 & 1 & 1 & 0}"] \& \Q^4 \ar[r,"\spmat{0 & 0 & 0 & 0 \\ -1 & 0 & 1 & 0 \\ -1 & 0 & 1 & 0 \\ 0 & -1 & 0 & 1}"] \& \Q^4
}
Next, we calculate the differential on the associated graded complex:
\cd[row sep=huge, column sep=huge]{
\Q^4 \ar[r,"\spmat{1 & 0 & 0 & 0 \\ 0 & 0 & 0 & 0 \\ 1 & 0 & 0 & 0 \\ 0 & 0 & 0 & 0}"] \& \Q^4 \ar[r,"\spmat{0 & 0 & 0 & 0 \\ -1 & 0 & 1 & 0 \\ 0 & 0 & 0 & 0 \\ 0 & -1 & 0 & 1}"] \& \Q^4
}
As in the annular case, this can be thought of as removing all parts of the differential that lower the $\g^\Sigma$-grading. For example, in the original complex, we have that
\eq{
    \diff(\underset{v_+}{\multiHopfLinkab}) = \underset{v_+ \tensor v_-}{\multiHopfLinkbb} + \underset{v_- \tensor v_+}{\multiHopfLinkbb}
}
but in the associated graded complex, we have
\eq{
    \diff(\underset{v_+}{\multiHopfLinkab}) = \underset{v_+ \tensor v_-}{\multiHopfLinkbb}
}
We can calculate that:
\eq{
    \MKh^{-2}(L) &\iso \Q^2\{-4\} \directsum \Q\{-6\} \\
    \MKh^{-1}(L) &\iso \Q\{-4\} \\
    \MKh^0(L) &\iso \Q \directsum \Q\{-2\}
}
Using subscripts to denote the $\g^\Sigma$-grading, we can refine this further to:
\eq{
    \MKh^{-2}(L) &\iso \Q_{(1,-1)}\{-4\} \directsum \Q_{(-1,1)}\{-4\} \directsum \Q_{(-1,-1)}\{-6\} \\
    \MKh^{-1}(L) &\iso \Q_{(-1,-1)}\{-4\} \\
    \MKh^0(L) &\iso \Q_{(1,1)} \directsum \Q_{(-1,-1)}\{-2\}
}

\end{example}


\section{Spectral sequences}
\label{sec:spectral-sequences}

One notable fact about annular Khovanov homology is that, given a link $L$, there is a spectral sequence with $E_1$ page isomorphic to $\AKh(L)$ that converges to $\Kh(L)$. It is natural to ask if our homology fits into similar spectral sequences as well.

Let $\Sigma$ be an $n$-punctured disk, let $L$ be a link in $\Sigma$, and let $\Sigma'$ be a disk with some subset of $m \le n$ of these punctures. We can view $L$ as a link in $\Sigma'$ by the natural inclusion map $\Sigma \inj \Sigma'$. Note that $\MKh(L \subset \Sigma)$ is naturally a $H_1(\Sigma; \Z)$-graded module, and $\MKh(L \subset \Sigma')$ is likewise graded by $H_1(\Sigma'; \Z)$. In order to use the classical results about spectral sequences to compare these two, we will first turn them into $\Z$-graded modules. Let $\{v_1,v_2,\dots,v_n\}$ be the basis of $H_1(\Sigma)$ where $v_i$ corresponds to a loop around the $i$-th puncture. Let $\epsilon: H_1(\Sigma; \Z) \to \Z$ be the map of posets that sends an element to the sum of its components:
\begin{equation}\label{eq:epsilon}
    \epsilon: \sum_{1 \le i \le n} c_iv_i \mapsto \sum_{1 \le i \le n} c_i
\end{equation}

Let $D$ be a diagram for $L$. We have a $H_1(\Sigma;\Z)$-grading $\g^\Sigma$ defined on generators of $\CKh(D)$ above in \cref{eq:grading}; taking the  composition $\epsilon \comp \g^\Sigma$ gives us a $\Z$-grading. Since the grading $\g^\Sigma$ induced a filtration on $\CKh(D)$ as a chain complex, and $\epsilon$ is a map of posets, we have that $\epsilon \comp \g^\Sigma$ induces a $\Z$-filtration on $\CKh(D)$. We will denote this filtration by $F^\Sigma$ (a flatter $\filt$ for a ``flattened'' filtration). Similarly, we can define a map $\epsilon': H_1(\Sigma';\Z) \to \Z$ and a corresponding induced filtration $F^{\Sigma'}$ on $\CKh(D)$.

First, we need a lemma to help us compare these two filtrations on $\CKh(D)$:

\begin{lemma}\label{lemma:two-filtrations}
Let $\filt^1, \filt^2$ be two $\poset$-filtrations on a module $M$. Then $\filt^1$ induces a $\poset$-filtration on the associated graded module of $M$ with respect to $\filt^2$. Additionally, if the sublattice of $\submodules(M)$ generated by $\filt^1$ and $\filt^2$ is distributive, then this new filtration is also distributive.
\end{lemma}

\begin{proof}
First, we describe the filtration that $\filt_{-}^1$ induces on $\filt_b^2 M$:
\begin{equation*}
    \filt_a^1 (\filt_b^2 M) \defeq \filt_a^1 M \intersect \filt_b^2 M
\end{equation*}
Let $\pi: \filt_b^2 M \to \gr[b]^2 M$ be the natural quotient map. We can use $\pi$ to transfer the induced filtration onto the associated graded object:
\begin{equation*}
    \filt_p^1(\gr[q]^2 M) \defeq \pi(\filt_p^1 (\filt_q^2 M)) = \frac{\filt_p^1 (\filt_q^2 M) + \filt_{q-1}^2 M}{\filt_{q-1}^2 M}
\end{equation*}
We can then take the direct sum of each $\gr[b](M)$ to get the desired filtration on $\gr(M)$.

If $\filt^1$ and $\filt^2$ distribute, then
\begin{equation*}
    - \intersect \filt_b^2: \submodules(M) \to \submodules(\filt_b^2)
\end{equation*}
and
\begin{equation*}
    \submodules(\pi): \submodules(\filt_b^2 M) \to \submodules(\gr[b](M))
\end{equation*}
are lattice homomorphisms, and therefore preserve distributive sublattices. Therefore, $\filt_{-}^1 \gr[b](M) = \submodules(\pi)(\filt_{-}^1 M \intersect \filt_b^2 M)$ will be distributive as well.
\end{proof}

The above lemma also applies to modules with extra structure, for example chain complexes. It turns out that, if two filtrations on a chain complex are suitably ``compatible'', we get a spectral sequence between the homologies of their associated graded complexes:

\begin{lemma}\label{lemma:two-filtrations-spectral-sequence}
Let $\filt^1, \filt^2$ be two $\Z$-filtrations on a chain complex $C$, and let $\gr^1, \gr^2$ denote their associated graded complexes. If $\gr^1(\gr^2(C)) \iso \gr^1(C)$ as graded complexes, then there is a spectral sequence with $E_1$ page isomorphic to $H_*(\gr^1(C))$ that converges to $\gr^1(H_*(\gr^2(C)))$.
\end{lemma}

\begin{proof}
Consider the chain complex $\gr^2(C)$ with filtration induced by $\filt^1$. The spectral sequence of this filtered complex has $E_1$ page isomorphic to
\begin{equation*}
    E_1 \iso H_*(\gr^1(\gr^2(C))) \iso H_*(\gr^1(C))
\end{equation*}
Additionally, the spectral sequence converges to
\begin{equation*}
    E_\infty \iso \gr^1(H_*(\gr^2(C))) \qedhere
\end{equation*}
\end{proof}

Therefore, if we could prove that $\gr^\Sigma(\gr^{\Sigma'}(C)) \iso \gr^\Sigma(C)$, then we would have a spectral sequence relating our two homologies. This last lemma will help with proving that:

\begin{lemma}\label{lemma:associated-graded}
Let $\filt^1, \filt^2$ be two filtrations on a chain complex $(C,\diff)$ over a field $\field$, and let $\g^1, \g^2$ be their respective gradings on $C$. If, for any $x,y\in C$ with $y$ contained in $\diff(x)$, we have that $\g^1(y)=\g^1(x) \implies \g^2(y)=\g^2(x)$, then there is an isomorphism of $\Z$-graded chain complexes $\gr^1 ( \gr^2 (C)) \iso \gr^1 (C)$.
\end{lemma}

\begin{proof}
Consider $C$ as a $(\g^1,\g^2)$-bigraded module. We already know that $\gr(M)\iso M$ for any filtered vector space $M$, so we have that $\gr^1(\gr^2(C)) \iso \gr^1(\CKh(D))$ as modules. The interesting part is proving the analogous fact in the presence of a differential.

Our bigrading allows us to decompose $\diff$ as:
\eq{
    \diff = \sum_{i, j} \diff^{(i,j)}
}
where $\diff^{(i.j)}$ is the summand of $\diff$ that changes the $\g^1$-grading by $i$ and the $\g^2$-grading by $j$. We can use this notation to express the differentials on the associated graded complexes as well:
\eq{
    \gr^1(\diff) &= \sum_{j} \diff^{(0,j)} = \diff^{(0,0)} \\
    \gr^2(\diff) &= \sum_{i} \diff^{(i,0)} = \sum_{i} \diff^{(i,0)}
}
Therefore, we see that $\gr^1(\diff)$ is a summand of $\gr^2(\diff)$, so
\eq{
\gr^1(\gr^2(\diff)) = \diff^{(0,0)} = \gr^1(\diff)
}
This allows us to conclude that $\gr^1(\gr^2(\CKh(D))) \iso \gr^1(\CKh(D))$ as $\Z$-graded chain complexes with respect to the $\g^1$-grading.
\end{proof}

With this in mind, we can now relate $\MKh(L\subset 
\Sigma\cross \I)$ and $\MKh(L \subset \Sigma'\cross \I)$.

\begin{proof}[Proof of \cref{thm:spectral-sequence}]
We would like to verify that $F^\Sigma$ and $F^{\Sigma'}$ satisfy the conditions of \cref{lemma:associated-graded}. Recall that, with respect to a single puncture, the Khovanov differential either preserves the $k$-grading (equivalently, the $\g^A$-grading), or lowers it by $2$. Therefore, $\diff$ can lower the $(\epsilon\comp\g^{\Sigma})$-grading by $0,2,4,\dots$, and can lower the $(\epsilon\comp\g^{\Sigma'})$-grading by at most that amount. If $x,y\in \CKh(D)$ are such that $y$ is contained in $\diff(x)$, we have that
\begin{equation*}
(\epsilon\comp\g^\Sigma)(y) - (\epsilon\comp\g^\Sigma)(x) \le (\epsilon\comp\g^{\Sigma'})(y) - (\epsilon\comp\g^{\Sigma'})(x) \le 0
\end{equation*}
and therefore
\begin{equation*}
(\epsilon\comp\g^\Sigma)(y) - (\epsilon\comp\g^\Sigma)(x) = 0\implies (\epsilon\comp\g^{\Sigma'})(y) - (\epsilon\comp\g^{\Sigma'})(x) = 0
\end{equation*}

Now we can use \cref{lemma:associated-graded} combined with \cref{lemma:two-filtrations-spectral-sequence} to conclude that there exists a spectral sequence with
\begin{equation*}
    E_1 \iso H^*(\gr^{\Sigma}(\CKh(D))) \iso \MKh(L \subset \Sigma\cross \I)
\end{equation*}
that converges to
\begin{equation*}
    E_\infty \iso \gr^{\Sigma} H^*(\gr^{\Sigma'}(\CKh(D))) \iso \MKh(L \subset \Sigma' \cross \I)
\end{equation*}
where both homologies are graded by $\epsilon \comp \g^\Sigma$.
\end{proof}


\section{Relationships with other link homology theories}
\label{sec:relationships}

\subsection{(Annular) Khovanov homology}
\label{sec:akh-connection}

We can use the spectral sequences defined in \cref{sec:spectral-sequences} to draw a few different connections to Khovanov homology and annular Khovanov homology. First, we identify special cases where $\MKh(L)$ is already isomorphic to a previously-defined link invariant.

\begin{lemma}\label{lemma:khovanov}
Let $\D$ be a disk, and let $D$ be a diagram for a link $L \subset \D \cross \I$. Then $\MKh(L \subset \D \cross \I) \iso \Kh(L)$.
\end{lemma}

\begin{proof}
Since $H_1(\D;\Z) \iso 0$, $\g^\D$ is a trivial grading, and therefore
\begin{equation*}
    \MKh(L \subset \D \cross \I) = H^*(\gr^\D(\CKh(D))) \iso H^*(\CKh(D)) = \Kh(L)
    \qedhere
\end{equation*}
\end{proof}

\begin{lemma}\label{lemma:annular-khovanov}
Let $A$ be an annulus, and let $L \subset A\cross \I$ be an annular link. Then $\MKh(L \subset A\cross \I) \iso \AKh(L)$.
\end{lemma}

\begin{proof}
Since $H_1(A;\Z) \iso \Z$, we get that the $k$- and $\g^A$-gradings are identical. Therefore, for a diagram $D$ of $L$:
\begin{equation*}
    \MKh(L\subset A\cross \I) = H^*(\gr^A(\CKh(D)) \iso H^*(\gr^k(\CKh(D))) = \AKh(D)
    \qedhere
\end{equation*}
\end{proof}

As a first application of \cref{thm:spectral-sequence}, we get a spectral sequence from $\MKh(L)$ to $\Kh(L)$:

\begin{proof}[Proof of \cref{thm:mkh-spectral-sequence}]
This follows immediately from \cref{lemma:khovanov} as a special case of \cref{thm:spectral-sequence} with $\Sigma = \Sigma$ and $\Sigma' = \D$ a disk.
\end{proof}

Additionally, this lets us derive the known spectral sequence from annular Khovanov homology to Khovanov homology as a special case of our construction.

\begin{corollary}\label{cor:annular-spectral-sequence}
For any annular link $L$, there is a spectral spectral sequence with $E_1$ page isomorphic to $\AKh(L)$ that converges to $\Kh(L)$.
\end{corollary}

\begin{proof}
This follows immediately from \cref{lemma:annular-khovanov} as a special case of \cref{thm:spectral-sequence} with $\Sigma = A$ an annulus and $\Sigma' = \D$ a disk.
\end{proof}

We get a second connection to annular Khovanov homology when $\Sigma$ is a disk with multiple punctures, and we single out a particular puncture.

\begin{proof}[Proof of \cref{thm:akh-spectral-sequence}]
    This is almost like a dual case to \cref{cor:annular-spectral-sequence}. The proof follows from \cref{lemma:annular-khovanov} as a special case of \cref{thm:spectral-sequence} with $\Sigma = \Sigma$ and $\Sigma' = A$.
\end{proof}

\subsection{APS homology}
\label{sec:aps-relationship}

``APS homology'' is the name we will use to refer to the invariant of links in thickened $\I$-bundles as defined in \cite{asaeda2004}; we will denote it $\APS(L)$. Annular Khovanov homology originated as a specialization of this theory; since our homology is a generalization of annular Khovanov homology, one might expect there to be connections between our homology and APS homology.

As originally defined, $\APS(L)$ is an invariant of framed links, and therefore is not invariant under Reidemeister I moves. The homological and quantum gradings also use different conventions than we did in defining $\MKh(L)$. We can reconcile both of these differences by defining another invariant $\wt{\APS}(L)$ that is isomorphic to APS homology as a module, but is an invariant of (unframed) links and follows the same grading conventions we have used thus far.

Let $D \subset \Sigma$ be a link diagram. Let $C(\Sigma)$ denote the set of all homotopy classes of unoriented, non-trivial simple closed curves in $\Sigma$. Let $x\in\CKh(D)$ be a generator representing a complete resolution of $D$ with signed circles $\set{(C_i, s_i)}_i$. We define a grading on $\CKh(D)$ as:
\eq{
    \Phi(x) \defeq \sum_{i} s_i \bvec{C_i} \qquad \in \Z C(\Sigma) 
}

Note that the definition of $\Phi$ is very similar to that of $\g^\Sigma$, the difference being that $\Phi$ takes values in $Z C(\Sigma)$ instead of $H_1(\Sigma;\Z)$. We would like to define a variant of APS homology as the homology of the associated graded complex of $\CKh(D)$ with respect to $\Phi$; however, $\Phi$ does not induce a filtration on $\CKh(D)$ with the product partial order on $\Z C(\Sigma)$, so we need to work around this.

Let $\epsilon: \Z C(\Sigma) \to \Z$ be the map of posets defined by summing the components analogously to \cref{eq:epsilon}, i.e.
\eq{
    \epsilon: \sum_i s_i \bvec{C_i} \mapsto \sum_i s_i
}
Define a new partial order $\trianglelefteq$ on $\Z C(\Sigma)$, where $\alpha \trianglelefteq \beta$ if and only if either $\alpha = \beta$ or $\epsilon(\alpha) < \epsilon(\beta)$.

\begin{lemma}\label{lemma:aps-filtered}
The Khovanov differential is filtered with respect to $(\Z C(\Sigma), \trianglelefteq)$.
\end{lemma}

\begin{proof}
Let $D$ be a link diagram in $\Sigma$, and let $x,y \in \CKh(D)$ be generators such that $\diff(x)$ contains a non-zero multiple of $y$ as a summand. We have a few cases to consider:
\begin{itemize}
    \item If $x$ and $y$ are related by a merge or split involving only non-trivial circles, then $\epsilon(\Phi(y)) = \epsilon(\Phi(x)) - 1$ since we know that $\diff$ preserves the $q$-grading. Therefore, $\epsilon(\Phi(y)) < \epsilon(\Phi(x))$.
    \item If $y$ is obtained from $x$ by splitting a trivial circle into two non-trivial circles, then it must be that both non-trivial circles represent the same class in $C(\Sigma)$. Therefore, either $\Phi(x) = \Phi(y)$ (if the trivial circle is marked with a $+$) or $\epsilon(\Phi(y)) = \epsilon(\Phi(x)) - 2$ (if the trivial circle is marked with a $-$).
    \item If $y$ is obtained from $x$ by splitting a non-trivial circle into a trivial circle and a non-trivial circle, then both non-trivial circles represent the same class in $C(\Sigma)$, so $\Phi(x) = \Phi(y)$.
    \item If $y$ is obtained from $x$ by merging a non-trivial circle and a trivial circle, then the resulting circle is non-trivial, and represents the same class in $C(\Sigma)$, so $\Phi(x) = \Phi(y)$.
    \item If $y$ is obtained from $x$ by merging two non-trivial circles to obtain a trivial circle, then the two non-trivial circles must represent the same class in $C(\Sigma)$, so $\Phi(x) = \Phi(y)$.
    \item In all other cases, $\Phi(x) = \Phi(y)$. \qedhere
\end{itemize}
\end{proof}

Now, we have the language needed to define our variant of APS homology.

\begin{definition}
Let $D$ be a diagram for a link $L$ in $\Sigma \cross \I$. Regard $\CKh(D)$ as a $\Z C(\Sigma)$-filtered module under the partial order $\trianglelefteq$. Define
\eq{
    \wt{\APS}(L) \defeq H^*(\gr^\Phi(\CKh(D)))
}
\end{definition}

\begin{proposition}[{c.f.\ \cite[Theorem 6.2]{asaeda2004}}]
$\wt{\APS}(L)$ is an invariant of links in $\Sigma \cross \I$ up to ambient isotopy.
\end{proposition}

\begin{proof}
For $S=\set{(C_i, s_i)}_i$ a resolution of a diagram $D$, the original gradings on APS homology\footnote{The original definition used the opposite convention regarding $+$ and $-$ labels on circles, as well as different terminology used to describe resolutions.} are
\begin{itemize}
    \item $I(S) \defeq \#\set{\text{0-resolved crossings}} - \#\set{\text{1-resolved crossings}}$
    \item $J(S) \defeq I(S) + 2\tau(S)$, \newline where $\tau(S) \defeq \#\set{\text{trivial circles labeled $-$}} - \#\set{\text{trivial circles labeled $+$}}$
    \item $\Psi(S) \defeq \sum_{i} (-s_i) \bvec{C_i}$
\end{itemize}

We can match these with our gradings on $\wt{\APS}(L)$ as follows:
\begin{itemize}
    \item $\gr[h](S) = \frac{1}{2}(n-I(S))$
    \item $\gr[q](S) = -(\epsilon(\Psi(S))+\tau(S))+\gr[h](S)+\writhe(L)$
    \item $\Phi(S) = -\Psi(S)$
\end{itemize}
In \cite{asaeda2004}, it was proven that $\APS(L)$ is an invariant of framed links, with Reidemeister I moves inducing isomorphisms of modules with fixed grading shifts. With our change of gradings, $\wt{\APS}(L)$ is actually an invariant of unframed links.
\end{proof}

Fix an ordering on the basis of $\Z C(\Sigma)$, and consider the ring of Laurent polynomials $\Z[q^{\pm 1}, x_{c_1}^{\pm 1}, x_{c_2}^{\pm 1}, \dots]$ for $c_i \in C(\Sigma)$. When $\Sigma$ is implied, we will denote the above using the shorthand $\Z[q^{\pm 1}, x^{\pm 1}]$. Additionally, given a vector $v=(v_1,v_2,\dots) \in \Z C(\Sigma)$, let

\begin{equation*}
    x^v \defeq \prod_i x_{c_i}^{v_i} = x_{c_1}^{v_1}x_{c_2}^{v_2}\dots
\end{equation*}

Similarly, we will write $\Z[q^{\pm 1}, y^{\pm 1}]$ to denote the ring $\Z[q^{\pm 1}, y_{c_1}^{\pm 1}, y_{c_2}^{\pm 1}, \dots, y_{c_n}^{\pm 1}]$ for a basis $\set{c_i}_{i=1}^n$ of $H_1(\Sigma;\Z)$, and define $y^v$ analogously for $v \in H_1(\Sigma;\Z)$.

\begin{definition}
Define the Euler characteristics of our homologies as follows:
\eq{
    \chi(\wt{\APS}(L)) &\defeq \sum_{i,j,v} (-1)^i q^j x^v \rk (\wt{\APS}^{i,j,v}(L)) \qquad &\in \Z[q^{\pm 1}, x^{\pm 1}] \\
    \chi(\MKh(L)) &\defeq \sum_{i,j,v} (-1)^i q^j y^v \rk (H^{i,j,v}(L)) \qquad &\in \Z[q^{\pm 1}, y^{\pm 1}]
}
\end{definition}

Define $h: \Z C(\Sigma) \to H_1(\Sigma;\Z)$ on generators as $h(c) = [c]$, and extend linearly. We can then see that $h$ induces a map between these two Euler characteristics:

\begin{proof}[Proof of \cref{thm:euler-characteristic}]
Define the map $\chi_h$ by letting $\chi_h(q) = q$ and $\chi_h(x_c) = y_{[c]}$, then extending algebraically. The proof that $\chi_h(\chi(\wt{\APS}(L))) = \chi(\MKh(L))$ follows from the fact that both homologies come from filtrations on the same chain complex.
\end{proof}

We actually have a stronger connection between $\wt{\APS}(L)$ and $\MKh(L)$. In the same vein as \cref{sec:spectral-sequences}, we can construct a spectral sequence between the two homologies. As before, we will ``flatten'' the filtrations so that we can use the standard spectral sequence construction\footnote{Morally, this should probably also be thought of as induced by $h$, but flattening the gradings obscures this relationship.}.

For a link diagram $D \subset \Sigma$, let $F^\Phi$ be the $\Z$-filtration induced by the $\Z$-grading $\epsilon \comp \Phi$ on $\CKh(D)$. Similarly, let $F^\Sigma$ be the $\Z$-filtration induced by the $\Z$-grading $\epsilon \comp \g^\Sigma$ on $\CKh(D)$. We can prove the analogous statement to \cref{lemma:associated-graded} for these filtrations.

\begin{lemma}\label{lemma:aps-associated-graded}
Let $\gr^\Phi$ denote the associated graded object with respect to $F^\Phi$, and define $\gr^\Sigma$ analogously. Then there is an isomorphism of $\Z$-graded chain complexes $\gr^\Phi ( \gr^\Sigma \CKh(D)) \iso \gr^\Phi \CKh(D)$.
\end{lemma}

\begin{proof}
Our filtrations $F^\Phi$ and $F^\Sigma$ come from two gradings $\Phi$ and $\g^\Sigma$ (respectively) on $\CKh(D)$, so we can consider $\CKh(D)$ as a $(\Phi,\g^\Sigma)$-bigraded module. We already know that $\gr(M)\iso M$ for any filtered vector space $M$, so we have that $\gr^\Phi(\gr^\Sigma(\CKh(D))) \iso \gr^\Phi(\CKh(D))$ as modules. The interesting part is proving the analogous fact for the differential.

Let $\diff$ denote the Khovanov differential on $\CKh(D)$. We would like to show that, for any generators $x,y \in \CKh(D)$ with $y$ contained in $\diff(x)$, if $(\epsilon\comp\Psi)(y) - (\epsilon\comp\Psi)(x) = 0$, then $(\epsilon\comp\g)(y) - (\epsilon\comp\g)(x) = 0$. This requires checking the same cases as in the proof of \cref{lemma:aps-filtered} where $\Phi(x) = \Phi(y)$. It turns out that in every case, we get two circles with the same class in $C(\Sigma)$ cancelling each other out. Additionally, if $C_1$ and $C_2$ are two circles representing the same class in $C(\Sigma)$, then they must also represent the same class in $H_1(\Sigma)$. Therefore, we get that $(\epsilon\comp\g)(y) - (\epsilon\comp\g)(x) = 0$ as well.

Since $(\epsilon\comp\Psi)(y) - (\epsilon\comp\Psi)(x) = 0$ implies that $(\epsilon\comp\g)(y) - (\epsilon\comp\g)(x) = 0$, we get that $\gr ( \gr' \CKh(D)) \iso \gr \CKh(D)$.
\end{proof}

Finally, we have all the necessary pieces together to prove \cref{thm:aps-spectral-sequence}.

\begin{proof}[Proof of \cref{thm:aps-spectral-sequence}]
Our spectral sequence is induced by the filtered complex $(\gr'(\CKh(D)), \filt)$. The $E_0$ page of this spectral sequence is $\gr(\gr' \CKh(D))$, which is isomorphic to $\gr(\CKh(D))$ as a chain complex by \cref{lemma:aps-associated-graded}. Therefore, by \cref{lemma:two-filtrations-spectral-sequence}, we have a spectral sequence with
\eq{
E_1 \iso H^*(\gr(\CKh(D)) \iso \wt{\APS}(L)
}
that converges to the underlying module of
\begin{equation*}
    E_\infty \iso \gr H^*(\gr' \CKh(D)) \iso H^*(\gr'\CKh(D)) \iso \MKh(L \subset \Sigma \cross \I) \qedhere
\end{equation*}
\end{proof}


\section{Invariance}
\label{sec:invariance}

In this section, we finish what we started in \cref{sec:mkh} by providing the proofs that $\MKh(L)$ is a well-defined link invariant, thereby proving \cref{thm:mkh-invariance}.

\begin{remark}\label{remark:total-complexes}
One nice feature about chain complexes arising from a cube of resolutions is that they can naturally be presented as complexes of complexes (of complexes of\dots, etc.\ ) Throughout this section, we will often implicitly identify double complexes with their total complexes and vice versa. It is worth pointing out that this trick still works when dealing with filtered objects.

Specifically, we note that a chain complex of filtered modules is the same thing as a filtered chain complex. This allows us to think about chain complexes of filtered chain complexes as filtered double complexes, which we then flatten into filtered chain complexes.
\end{remark}


\subsection{Reidemeister I invariance}
\label{sec:proof-r1-invariance}

\begin{remark}
For the next three subsections, we aim to mimic Bar-Natan's argument for the invariance of Khovanov homology (see \cite{bar-natan2005}), while checking some extra conditions to ensure that the our filtration on the Khovanov complex is also suitably invariant. For some diagram $D$, we will use the notation $\complex{D}$ throughout to denote the filtered chain complex $\CKh(D_\Sigma)$, and we will often describe complexes as complexes of complexes as in \cref{remark:total-complexes}.
\end{remark}

\subsubsection{Unfiltered version}
\label{sec:proof-r1-invariance-unfiltered}

We would like to show that the complexes $\complex{\rI}$ and $\complex{\rIb}$ have the same homology. Let
\eq{
C = \complex{\rI} = \left( \complex{\rIa} \to*{m} \complex{\rIb}\{1\} \right)
}
and let
\eq{
C' = \left( \complex{\rIa}_{v_+} \to*{m} \complex{\rIb}\{1\} \right) \sub C
}
The notation $\complex{\dots}_{v_+}$ denotes the subcomplex consisting of all generators represented by diagrams in which the loop pictured in the diagram is labeled with a $+$. Note that, since $v_+$ is a ``unit'' for the merge $m$, it follows that $m$ is an isomorphism in $C'$, so $C'$ is acyclic. This means that the quotient map $q: C\to C/C'$ is a quasi-isomorphism. Consider the quotient
\eq{
C/C' = \left( \complex{\rIa}_{/v_+=0} \to 0 \right)
}

The notation $\complex{\dots}_{/v_+}$ denotes the quotient by all generators represented by diagrams in which the loop is labeled with a $+$. We want to construct a map $f: C/C' \to \complex{\rIb}$. On generators:
\eq{
    f : \underset{x \tensor v_-}{\rIa} \mapsto \underset{x}{\rIb}
}
Since all generators in $\complex{\rIa}_{/v_+=0}$ have the closed component marked with a $-$, we can see that $f$ is an isomorphism. We can then combine this with $q$ above to get a map $\rho_I: C = \complex{\rI} \to \complex{\rIb}$, defined as $\rho_I = f \compose q$. Since $q$ is a quasi-isomorphism and $f$ is an isomorphism, we have shown that $\rho_I$ is a quasi-isomorphism as well, thus completing the proof of Reidemeister I invariance.

\subsubsection{Filtered version}
\label{sec:proof-r1-invariance-filtered}

We need to check that $\rho_I: \complex{\rI} \to \complex{\rIb}$ is a filtered quasi-isomorphism. To help, we start by labelling components of $\Sigma \setminus D$ as below:
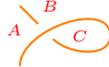
\eq{
    \rIregions
}
Extend these component labels to any other diagram obtained as a resolution of this one:
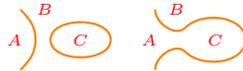
\begin{center}\begin{tabular}{cc}
      \rIaregions & \rIbregions
\end{tabular}\end{center}
Recall that $\rho_I = f \comp q$. We will proceed to prove that each of the factors $f$ and $q$ are filtered quasi-isomorphisms.

The map $q: C \to C/C'$ is strict, since we defined it via a quotient by a subcomplex. Therefore, we have a short exact sequence
\cd{
0 \ar[r] \& \gr(C') \ar[r] \& \gr(C) \ar[r,"\gr(q)"] \& \gr(C/C') \ar[r] \& 0
}
To verify that $q$ is a filtered quasi-isomorphism, it therefore suffices to show that $\gr(C')$ is acyclic; the long exact sequence induced by the above would then show that $\gr(q)$ induces an isomorphism $H^*(\gr(C)) \to H^*(\gr(C/C'))$. Above, we showed that $C'$ is acyclic by noticing that the differential is given by a merge map $m$ that restricts to an isomorphism. By looking at the diagrams representing the source and target of $m$, we see that $m$ leaves the grading with respect to punctures in $B$ unchanged, and combines the gradings with respect to punctures in $A$ and $C$. We only need to consider Reidemeister I moves that represent smooth isotopy in $\Sigma \cross \I$; therefore, we can assume that there are no punctures in $C$, and thus that $m$ preserves the grading of all generators. This implies that $\gr(m)$ is also an isomorphism, and thus that $H^*(\gr(C'))=0$. This completes the proof that $q$ is a filtered quasi-isomorphism.

We have already noted above that $f$ is an isomorphism. From the definition of $f$, we see that $f$ leaves the filtration with respect to punctures in $A$ and $B$ unchanged, as we assume there are no punctures in $C$. Therefore, $f$ is a filtered isomorphism, which therefore implies that it is a filtered quasi-isomorphism.

This completes the proof that $\rho_I: \complex{\rI} \to \complex{\rIb}$ is a filtered quasi-isomorphism, and thus that $\MKh(L)$ is invariant under Reidemeister I moves.


\subsection{Reidemeister II invariance}
\label{sec:proof-r2-invariance}

\subsubsection{Unfiltered version}

We would like to show that the complexes $\complex{\rII}$ and $\complex{\rIIba}$ have the same homology. Let $C = \complex{\rII}$ be the complex
\begin{center}\begin{tikzcd}
\complex{\rIIab}\{1\} \ar[r,"m"] & \complex{\rIIbb}\{2\} \\
\complex{\rIIaa} \ar[r,"\diff_1"] \ar[u,"\Delta"] & \complex{\rIIba}\{1\} \ar[u,"\diff_2"]
\end{tikzcd}\end{center}
Note that $\diff_1$ and $\diff_2$ can either be merge or split maps, depending on the rest of the diagram. Let $C' \sub C$ be the subcomplex
\begin{center}\begin{tikzcd}
\complex{\rIIab}_{v_+}\{1\} \ar[r,"m"] & \complex{\rIIbb}\{2\} \\
0 \ar[r] \ar[u] & 0 \ar[u]
\end{tikzcd}\end{center}
Note that the merge map $m$ in $C'$ is an isomorphism, so $C'$ is acyclic.
Consider the quotient $C/C'$:
\begin{center}\begin{tikzcd}
\complex{\rIIab}_{/v_+=0}\{1\} \ar[r] & 0 \\
\complex{\rIIaa} \ar[r,"\diff_1"] \ar[u,"\Delta"] & \complex{\rIIba}\{1\} \ar[u]
\end{tikzcd}\end{center}
We can see that $C'' = \complex{\rIIba}\{1\}$ sits inside $C/C'$ as a subcomplex. Additionally, the split map $\Delta$ is an isomorphism in $C/C'$. Consider another subcomplex $C''' \sub C/C'$ consisting of all $\alpha \in \complex{\rIIaa}$ and all $(\beta, \tau \beta)$ in the graph of $\tau = \diff_1\Delta^{-1}$:
\cd{
    \beta \ar[r] \ar[dr,dashed,"\tau"] \& 0 \\
    \alpha \ar[r,"\diff_1"] \ar[u,"\Delta"] \& \tau \beta \ar[u]
}
The map $\Delta$ in $C'''$ is an isomorphism, so $C'''$ is acyclic.
Quotienting $C/C'$ by $C'''$ gives us the complex
\cd{
    \beta_{/\beta=\tau\beta} \ar[r] \ar[dr,dashed,"\tau"] \& 0 \\
    0 \ar[r] \ar[u] \& \gamma \ar[u]
}
Since all $\beta$ are identified with some $\gamma$, an element of the above complex is uniquely determined by a choice of $\gamma$, so we get the isomorphism $(C/C')/C''' \iso C''$. Since $C'$ and $C'''$ are acyclic, we find that $H^*(C) \iso H^*(C'')$ as desired.

\subsubsection{Filtered version}

We need to define a map $\rho_{II}: \complex{\rII} \to \complex{\rIIba}$ and check that this is a filtered quasi-isomorphism. We can define $\rho_{II}$ as $f \compose q_2 \compose q_1$, where $q_1: C \to C/C'$ and $q_2: C/C' \to (C/C')/C'''$ are the two quotient maps, and $f: (C/C')/C''' \to C''$ is the isomorphism implied above. We will proceed to prove that $q_1$, $q_2$, and $f$ are filtered quasi-isomorphisms. To help with this, label the regions of $\sigma \setminus L$ as below.
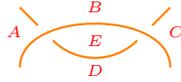
\eq{\rIIregions}
As before, extend these component labels to any other diagram obtained as a resolution of this one.

The proof that $q_1: C \to C/C'$ is a filtered quasi-isomorphism will be very similar to the case of the map $q$ in \cref{sec:proof-r1-invariance-filtered}. Note that $q_1$ is strict, as it is defined via quotient by a filtered submodule. We will prove that $q_1$ is a filtered quasi-isomorphism by proving that $\gr(C')$ is acyclic. Before, we proved that $C'$ is acyclic by noting that the differential is given by a merge $m$ which is an isomorphism. We can now observe that $m$ preserves the grading with respect to punctures in $A$, $B$, $C$, and $D$, and may change the grading with respect to punctures in $E$. Because we assume our Reidemeister II move represents smooth isotopy in $\Sigma \cross \I$, we can assume that there are no punctures in $E$, and therefore that $\gr(m)$ is an isomorphism. Therefore, $\gr(C')$ is acyclic, and thus $q_1$ is a filtered quasi-isomorphism.

Next, we will prove that $q_2: C/C' \to (C/C')/C'''$ is a filtered quasi-isomorphism. Again, $q_2$ is strict, so it suffices to prove that $\gr(C''')$ is acyclic. Note that the isomorphism $\Delta$ in $C'''$ preserves the grading with respect to punctures in $A$, $B$, $C$, and $D$, and may change the grading with respect to punctures in $E$. By the same reasoning as above, there are no punctures in $E$, so $\Delta$ is a graded isomorphism. Therefore, $\gr(C''')$ is acyclic.

Note that $f: (C/C')/C''' \to C''$ is already an isomorphism; the only barrier to it being a filtered isomorphism is that the generators of $\complex{\rIIab}_{/v_+=0}$ are identified with certain elements of $\complex{\rIIba}$ via $\tau = \diff_1\Delta^{-1}$. We can use \cref{quotient-grading} here, and check that for any generator represented by a diagram $\rIIba$, we cannot find a diagram $\rIIab$ with a lower $\g^\Sigma$-grading. Therefore, we will show that $f$ is a filtered isomorphism by showing that $\tau$ does not increase the $\g^\Sigma$-grading of any element. We know that the Khovanov differential $\diff_1$ does not increase the $\g^\Sigma$-grading (see \cref{lemma:proof-filtered-differential}), so it remains to check that $\Delta^{-1}: \complex{\rIIab}_{/v_+=0} \to \complex{\rIIaa}$ does not increase the $\g^\Sigma$-grading. We have already proven that $\Delta$ doesn't change the $\g^\Sigma$-grading, so it follows that $\Delta^{-1}$ does not either. This means that $f$ is a filtered isomorphism, and therefore a filtered quasi-isomorphism.

Since $\rho_{II} = f \comp q_2 \comp q_1$ and all three maps on the right-hand side are filtered quasi-isomorphisms, $\rho_{II}$ must be as well. Therefore, our homology $\MKh(L)$ is invariant under Reidemeister II moves.


\subsection{Reidemeister III invariance}
\label{sec:proof-r3-invariance}

\subsubsection{Unfiltered version}

We would like to show that the complexes $L = \complex{\rIII}$ and $R = \complex{\rIIIx}$ have the same homology. Written out (suppressing the $\complex{}$ notation and grading shifts), these complexes are:
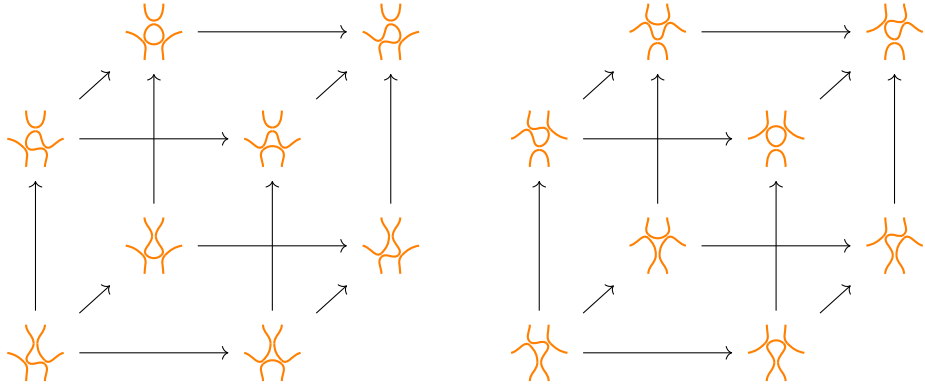
\begin{center}
\begin{tikzcd}[row sep=small, column sep=small]
& \rIIIbba \ar[rr] & & \rIIIbbb \\
\rIIIbaa \ar[rr] \ar[ur] & & \rIIIbab \ar[ur] & \\
& \rIIIaba \ar[rr] \ar[uu] & & \rIIIabb \ar[uu] \\
\rIIIaaa \ar[rr] \ar[ur] \ar[uu] & & \rIIIaab \ar[ur] \ar[uu] &
\end{tikzcd}
\qquad
\begin{tikzcd}[row sep=small, column sep=small]
& \rIIIxbba \ar[rr] & & \rIIIxbbb \\
\rIIIxbaa \ar[rr] \ar[ur] & & \rIIIxbab \ar[ur] & \\
& \rIIIxaba \ar[rr] \ar[uu] & & \rIIIxabb \ar[uu] \\
\rIIIxaaa \ar[rr] \ar[ur] \ar[uu] & & \rIIIxaab \ar[ur] \ar[uu] &
\end{tikzcd}
\end{center}
Note that the top level of $L$ is isomorphic to $\complex{\rIIIb}$ and the top level of $R$ is isomorphic to $\complex{\rIIIxb}$. These diagrams are related by two Reidemeister II moves, so we can re-use some of the constructions from \cref{sec:proof-r2-invariance} here.

Specifically, if we identify the top level of $L$ with the complex $C$ from the previous section, the definition of $C' \subseteq C$ gives us a subcomplex $L' \subseteq L$. We can also mimic the definition of $C''' \subseteq C/C'$ to get a subcomplex $L''' \subseteq L/L'$. Note that, since $L'$ and $L'''$ are acyclic, $(L/L')/L'''$ has the same homology as $L$. We can also repeat this process on the right side by identifying the top level of $R$ with $C$, and defining $R'$ and $R'''$ in the same way. This gives us two complexes $(L/L')/L'''$ and $(R/R')/R'''$:
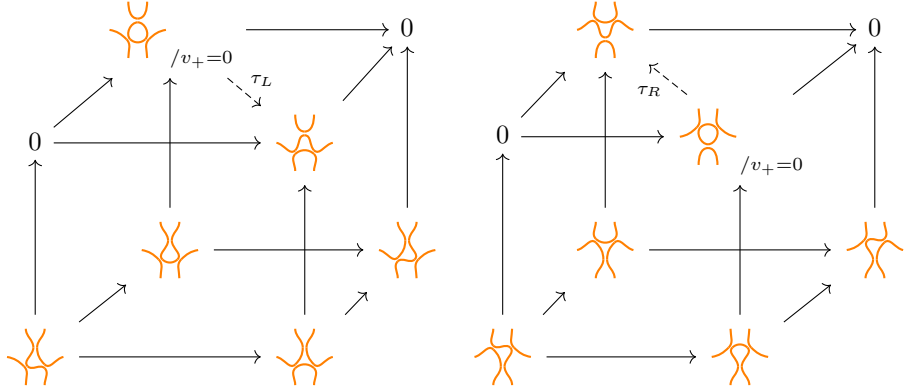
\begin{center}
\begin{tikzcd}[row sep=small, column sep=tiny]
& \rIIIbba_{/v_+=0} \ar[rr] \ar[dr,dashed,"\tau_L"] & & 0 \\
0 \ar[rr] \ar[ur] & & \rIIIbab \ar[ur] & \\
& \rIIIaba \ar[rr] \ar[uu] & & \rIIIabb \ar[uu] \\
\rIIIaaa \ar[rr] \ar[ur] \ar[uu] & & \rIIIaab \ar[ur] \ar[uu] &
\end{tikzcd}
\begin{tikzcd}[row sep=small, column sep=tiny]
& \rIIIxbba \ar[rr] & & 0 \\
0 \ar[rr] \ar[ur] & & \rIIIxbab_{/v_+=0} \ar[ur] \ar[ul,dashed,"\tau_R"] & \\
& \rIIIxaba \ar[rr] \ar[uu] & & \rIIIxabb \ar[uu] \\
\rIIIxaaa \ar[rr] \ar[ur] \ar[uu] & & \rIIIxaab \ar[ur] \ar[uu] &
\end{tikzcd}
\end{center}
Note that the bottom levels of $L$ and $R$ are isomorphic at each vertex since they're represented by isotopic diagrams. The two complexes above are then isomorphic via a map $\Upsilon: (L/L')/L''' \to (R/R')/R'''$ that preserves the bottom level, and transposes the top level via the isotopy $\rIIIbab \to \rIIIxbba$. Note that we don't need to define $\Upsilon$ on elements represented by diagrams like $\rIIIbba$ or $\rIIIxbab$ since these are already identified with other elements via the maps $\tau_L$ and $\tau_R$. Since we have shown that $L \homotopic (L/L')/L''' \iso (R/R')/R''' \homotopic R$, we conclude that $\Kh$ is invariant under Reidemeister III moves.

\subsubsection{Filtered version}
Since we have already verified in \cref{sec:proof-r2-invariance} that $\MKh(L)\iso \MKh((L/L')/L''')$ and $\MKh(R)\iso \MKh((R/R')/R''')$, all we need to check is that the isomorphism $\Upsilon$ is a filtered isomorphism. This isn't hard to see; $\Upsilon$ is defined by five planar isotopies of diagrams (four on the bottom and one on the top). We can assume that there are no punctures in the center region of the diagram, and each of these isotopies can be made to avoid the punctures in the six exterior regions. Therefore, $\MKh$ is invariant under Reidemeister III moves.

\begin{remark}
This, along with the previous two subsections, completes the proof that $\MKh(L)$ is a link invariant, which makes up the first half of \cref{thm:mkh-invariance}. The second half is the seemingly stronger statement that the filtered chain homotopy type of $\CKh(D;\field)$ is an invariant. This follows from the fact that two chain complexes of filtered projective modules are filtered quasi-isomorphic if and only if they are filtered chain homotopy equivalent (see \cite[\href{https://stacks.math.columbia.edu/tag/03TB}{Section 03TB}]{stacks-project}). Since $\CKh(D)$ is a free chain complex, and the filtration $\filt^\Sigma$ on $\CKh(D)$ is induced by a grading on generators, each associated graded complex $\gr[p]^\Sigma(\CKh(D))$ is a chain complex of free (thus projective) $\field$-modules. Therefore, given diagrams $D, D'$, the existence of a filtered quasi-isomorphism $\CKh(D \subset \Sigma) \to \CKh(D' \subset \Sigma)$ actually implies that the two complexes are filtered homotopy equivalent. We have constructed explicit filtered quasi-isomorphisms $\rho_I$ and $\rho_{II}$ realizing Reidemeister I and II moves on $\CKh(D)$ in \cref{sec:proof-r1-invariance-filtered} and \cref{sec:proof-r2-invariance}, respectively. Our proof of invariance for Reidemeister III moves does not produce a single quasi-isomorphism of free chain complexes, but does produce a zig-zag of such quasi-isomorphisms, which can be composed after realizing them as homotopy equivalences. Therefore, we have proved both halves of \cref{thm:mkh-invariance}.
\end{remark}

\section{Further examples}
\label{sec:further-examples}

Since $\MKh(L)$ is in some sense constructed as a generalization of annular Khovanov homology, it is natural to ask if it is actually entirely \textit{determined} by $\AKh(L)$ with respect to each individual puncture. This, however, is not the case; there exist non-trivial knots in punctured disks which become trivial unknots when any puncture is removed, and $\MKh(L)$ distinguishes some of these knots from the unknot. A natural class of such knots is studied in \cite{demaine2013}. We will work out $\MKh(L)$ for one such knot here.

\begin{example}
Let $L$ be the link in the twice-punctured disk depicted below.

$$\pictureHangingPuzzle$$

We calculate $\MKh(L)$ to be the following:
\eq{
     \MKh^0(L) &\iso \Q_{(-2,0)}\{-1\} \directsum \Q_{(0,0)}^2\{1\} \directsum \Q_{(0,-2)}\{-1\} \\
     \MKh^1(L) &\iso \Q_{(0,-2)}\{1\} \directsum \Q_{(-2,0)}\{1\} \directsum \Q_{(-2,-2)}\{-1\} \\
     \MKh^2(L) &\iso \Q_{(-2,-2)}\{1\}
}
\end{example}


\printbibliography

\end{document}